\documentclass[a4paper,leqno]{amsart}

\usepackage{latexsym}
\usepackage[english]{babel}
\usepackage{fancyhdr}
\usepackage[mathscr]{eucal}
\usepackage{amsmath}
\usepackage{mathrsfs}
\usepackage{amsthm}
\usepackage{amsfonts}
\usepackage{amssymb}
\usepackage{amscd}
\usepackage{bbm}
\usepackage{graphicx}
\usepackage{graphics}
\usepackage{latexsym}
\usepackage{color}

\newcommand{\ud}{\mathrm{d}}

\newcommand{\ii}{\mathrm{i}}

\newcommand{\C}{\mathbb C}

\newcommand{\R}{\mathbb R}

\theoremstyle{plain}
\newtheorem{theorem}{Theorem}[section]
\newtheorem{lemma}[theorem]{Lemma}
\newtheorem{corollary}[theorem]{Corollary}
\newtheorem{proposition}[theorem]{Proposition}

\theoremstyle{definition}
\newtheorem{definition}[theorem]{Definition}

\newtheorem{remark}[theorem]{Remark}
\newtheorem*{remark*}{Remark}

\numberwithin{equation}{section}

\begin{document}

\title[Singular Hartree equation in fractional perturbed Sobolev spaces]
{The singular Hartree equation in fractional perturbed Sobolev spaces}
\author[A.~Michelangeli]{Alessandro Michelangeli}
\address[A.~Michelangeli]{International School for Advanced Studies -- SISSA \\ via Bonomea 265 \\ 34136 Trieste (Italy).}
\email{alemiche@sissa.it}
\author[A.~Olgiati]{Alessandro Olgiati}
\address[A.~Olgiati]{International School for Advanced Studies -- SISSA \\ via Bonomea 265 \\ 34136 Trieste (Italy).}
\email{aolgiati@sissa.it}
\author[R.~Scandone]{Raffaele Scandone}
\address[R.~Scandone]{International School for Advanced Studies -- SISSA \\ via Bonomea 265 \\ 34136 Trieste (Italy).}
\email{rscandone@sissa.it}

\begin{abstract}
We establish the local and global theory for the Cauchy problem of the singular Hartree equation in three dimensions, that is, the modification of the non-linear Schr\"{o}dinger equation with Hartree non-linearity, where the linear part is now given by the Hamiltonian of point interaction. The latter is a singular, self-adjoint perturbation of the free Laplacian, modelling a contact interaction at a fixed point.
The resulting non-linear equation is the typical effective equation for the dynamics of condensed Bose gases with fixed point-like impurities. We control the local solution theory in the perturbed Sobolev spaces of fractional order between the mass space and the operator domain. We then control the global solution theory both in the mass and in the energy space.
\end{abstract}

\date{\today}

\subjclass[2000]{}

%

\maketitle


\section{The singular Hartree equation. Main results.}

The Hartree equation in $d$ dimension is the well-known semi-linear Schr\"{o}dinger equation with cubic convolutive non-linearity of the form
\begin{equation}\label{eq:general_Hartree}
\ii\partial_t u\;=\;-\Delta u + Vu+ (w*|u|^2)u
\end{equation}
in the complex-valued unknown $u\equiv u(x,t)$, $t\in\mathbb{R}$, $x\in\mathbb{R}^d$, for given  measurable functions  $V,w:\mathbb{R}^d\to\mathbb{R}$.

Among the several contexts of relevance of \eqref{eq:general_Hartree}, one is surely the quantum dynamics of large Bose gases, where particles are subject to an external potential $V$ and interact through a two-body potential $w$. In this case \eqref{eq:general_Hartree} emerges as the effective evolution equation, rigorously in the limit of infinitely many particles, of a many-body initial state that is scarcely correlated, say, $\Psi(x_1,\dots,x_N)\sim u_0(x_1)\cdots u_0(x_N)$, whose evolution can be proved to retain the approximate form $\Psi(x_1,\dots,x_N;t)\sim u(x_1,t)\cdots u(x_N,t)$ for some one-body orbital $u\in L^2(\mathbb{R}^d)$ that solves the Hartree equation \eqref{eq:general_Hartree} with initial condition $u(x,0)=u_0(x)$. The precise meaning of the control of the many-body wave function is in the sense of one-body reduced density matrices. The limit $N\to +\infty$ is taken with a suitable re-scaling prescription of the many-body Hamiltonian, so as to make the limit non-trivial. In the \emph{mean field} scaling, that models particles paired by an interaction of long range and weak magnitude, the interaction term in the Hamiltonian has the form $N^{-1}\sum_{j<k}w(x_j-x_k)$, and when applied to a wave function of the approximate form $u(x_1)\cdots u(x_N)$ it generates indeed the typical self-interaction term $(w*|u|^2)u$ of \eqref{eq:general_Hartree}. This scenario is today controlled in a virtually complete class of cases, ranging from bounded to locally singular potentials $w$, and through a multitude of techniques to control the limit (see, e.g., \cite[Chapter 2]{Benedikter-Porta-Schlein-2015} and the references therein).

The Cauchy problem for \eqref{eq:general_Hartree} is extensively studied and understood too, including its local and global well-posedness and its scattering -- for the vast literature on the subject, we refer to the monograph \cite{cazenave}, as well as to the recent work \cite{Miao-Hartree-2007}. Two natural conserved quantities for \eqref{eq:general_Hartree} are the mass and (as long as $w(x)=w(-x)$), the energy, namely,
\[
\begin{split}
\mathcal{M}(u)\;&=\;\int_{\mathbb{R}^d}|u|^2\,\ud x \\
\mathcal{E}(u)\;&=\;\frac{1}{2}\!\int_{\mathbb{R}^d}\big(|\nabla u|^2 + V|u|^2\big)\,\ud x+\frac{1}{4}\!\iint_{\substack{ \\ \\ \\ \!\!\!\!\!\!\!\!\!\!\mathbb{R}^d\times\mathbb{R}^d}}\!\!\!w(x-y)|u(x)|^2|u(y)|^2\,\ud x\,\ud y\,.
\end{split}
\]
The natural energy space is therefore $H^1(\mathbb{R}^d)$, and the equation is energy sub-critical for $w\in L^{1}(\mathbb{R}^d)+L^{\infty}(\mathbb{R}^d)$ and mass sub-critical for $w\in L^{q}(\mathbb{R}^d)+L^{\infty}(\mathbb{R}^d)$, for some $q\geqslant\max\{1,\frac d2\}$ ($q>1$ if $d=2$).

In fact, irrespectively of the technique to derive the Hartree equation from the many-body linear Schr\"{o}dinger equation (hierarchy of marginals, Fock space of fluctuations, counting of the condensate particles, and others), one fundamental requirement is that at least for the time interval in which the limit $N\to +\infty$ is monitored the Hartree equation itself is well-posed, which makes the understanding of the effective Cauchy problem an essential pre-requisite for the derivation from the many-body quantum dynamics.

In the quantum interpretation discussed above, the external potential $V$ can be regarded as a \emph{confining potential} or also as a \emph{local inhomogeneity} of the spatial background where particles are localised in, depending on the model. In general, as long as $V$ is locally sufficiently regular, this term is harmless both in the Cauchy problem associated to \eqref{eq:general_Hartree} and in its rigorous derivation from the many-body Schr\"{o}dinger dynamics. This, in particular, allows one to model local inhomogeneities such as `bump'-like impurities, but genuine `delta'-like impurities localised at some fixed points $X_1,\dots X_M\in\mathbb{R}^3$ certainly escape this picture.

In this work we are indeed concerned with a so-called `delta-like singular' version of the ordinary Hartree equation \eqref{eq:general_Hartree} where formally the local impurity $V(x)=\mathcal{V}(x-X)$ around the point $X$, for some locally regular potential $\mathcal{V}$ is replaced by $V(x)=\delta(x-X)$, and more concretely we study the Cauchy problem for an equation of the form 
\begin{equation}\label{eq:sin_Hartree_formal}
\ii\partial_t u\;=\;\textrm{``}-\Delta u + \delta(x-X)u\,\textrm{''}+ (w*|u|^2)u\,.
\end{equation}
There will be no substantial loss of generality, in all the following discussion, if we take one point centre $X$ only, instead of $X_1,\dots X_M\in\mathbb{R}^3$, and if we set $X=0$, which we will do throughout.

The precise meaning in which the linear part in the r.h.s.~of \eqref{eq:sin_Hartree_formal} has to be understood is the `\emph{singular Hamiltonian of point interaction}', that is, a singular perturbation of the negative Laplacian $-\Delta$ which, consistently with the interpretation of a local impurity that is so singular as to be supported only at one point, is a self-adjoint extension on $L^2(\mathbb{R}^d)$ of the symmetric operator $-\Delta|_{C^\infty_0(\mathbb{R}^d)}$, and therefore acts precisely as $-\Delta$ on $H^2$-functions supported away from the origin. In fact,  $-\Delta|_{C^\infty_0(\mathbb{R}^d)}$ is already essentially self-adjoint when $d\geqslant 4$, with operator closure given by the self-adjoint $-\Delta$ with domain $H^2(\mathbb{R}^d)$, therefore it only makes sense to consider the singular Hartree equation for $d\in\{1,2,3\}$, and the higher the co-dimension of the point where the singular interaction is supported, the more difficult the problem.

Our setting in this work will be with $d=3$. We shall comment later on analogous results in the simpler case $d=1$. In three dimensions one has the following standard construction, which we recall, for example, from \cite[Chapter I.1]{albeverio-solvable} and \cite[Section 3]{MO-2016}.

The class of self-adjoint extensions in $L^2(\mathbb{R}^3)$ of the positive and densely defined symmetric operator $-\Delta|_{C^\infty_0(\mathbb{R}^3\setminus\{0\})}$ is a one-parameter family of operators $-\Delta_\alpha$, $\alpha\in(-\infty,+\infty]$, defined by
\begin{equation}\label{eq:op_dom-opaction}
\begin{split}
\mathcal{D}(-\Delta_\alpha)\;&=\;\Big\{\psi\in L^2(\mathbb{R}^3)\,\Big|\,\psi=\phi_\lambda+\frac{\phi_\lambda(0)}{\alpha+\frac{\sqrt{\lambda}}{4\pi}}\,G_\lambda\textrm{ with }\phi_\lambda\in H^2(\mathbb{R}^3)\Big\} \\
(-\Delta_\alpha+\lambda)\,\psi\;&=\;(-\Delta+\lambda)\,\phi_\lambda\,,
\end{split}
\end{equation}
where $\lambda>0$ is an arbitrarily fixed constant and
\begin{equation}\label{eq:defGlambda}
G_\lambda(x)\;:=\;\frac{e^{-\sqrt{\lambda}\,|x|}}{4\pi |x|}
\end{equation}
is the Green function for the Laplacian, that is, the distributional solution to $(-\Delta+\lambda)G_\lambda=\delta$ in $\mathcal{D}'(\mathbb{R}^3)$.

The quadratic form of $-\Delta_\alpha$ is given by
\begin{equation}\label{eq:form_dom-opaction}
\begin{split}
\mathcal{D}[-\Delta_\alpha]\;&=\;H^1(\mathbb{R}^3)\dotplus\mathrm{span}\{ G_\lambda\} \\
(-\Delta_\alpha)[\phi_\lambda+\kappa_\lambda\,G_\lambda] \;&=\; -\lambda\|\phi_\lambda+\kappa_\lambda\,G_\lambda\|_2^2 \\
&\qquad+\|\nabla \phi_\lambda\|_2^2+\lambda\|\phi_\lambda\|_2^2+{\textstyle \big(\alpha+\frac{\sqrt{\lambda}\,}{4\pi}\big)}\,|\kappa_\lambda|^2\,.
\end{split}
\end{equation}

The above decompositions of a generic $\psi\in\mathcal{D}(-\Delta_\alpha)$ or $\psi\in\mathcal{D}[-\Delta_\alpha]$ are unique and are valid for every chosen $\lambda$.  The extension $-\Delta_{\alpha=\infty}$ is the Friedrichs extension and is precisely the self-adjoint $-\Delta$ on $L^2(\mathbb{R}^3)$ with domain $H^2(\mathbb{R}^3)$.

The operator $-\Delta_\alpha$ is reduced with respect to the canonical decomposition
\[
L^2(\mathbb{R}^3)\;\cong\; L^2_{\ell=0}(\mathbb{R}^3)\oplus\;\bigoplus_{\ell=1}^\infty  L^2_\ell(\mathbb{R}^3)
\]
in terms of subspaces $L^2_\ell(\mathbb{R}^3)$ of definite angular symmetry, and it is a non-trivial modification of the negative Laplacian in the spherically symmetric sector only, i.e.,
\begin{equation}\label{eq:DaS}
(-\Delta_{\alpha})|_{\mathcal{D}(-\Delta_{\alpha})\cap L^2_\ell(\mathbb{R}^3)}\;=\;(-\Delta)|_{H^2_\ell(\mathbb{R}^3)}\,,\qquad \ell\neq 0\,.
\end{equation}

Each
$\psi\in\mathcal{D}(-\Delta_\alpha)$ satisfies the short range asymptotics
\begin{equation}\label{eq:BPcontact}
\psi(x)\;=\;c_\psi\Big(\frac{1}{|x|}-\frac{1}{a}\Big)+o(1)\qquad\mathrm{as}\;\;x\to 0\,,\qquad a:=(-4\pi\alpha)^{-1}\,,
\end{equation}
or also, in momentum space,
\begin{equation}\label{eq:TMS_cond_asymptotics_1}
\int_{\substack{\,p\in\mathbb{R}^3 \\ \! |p|<R}}\,{\widehat \psi(p) \,\ud p}\;=\;d_\psi(R+2\pi^2\alpha)+o(1) \qquad\textrm{as}\qquad R\to +\infty\,,
\end{equation}
for some $c_\psi,d_\psi\in\mathbb{C}$. Equations \eqref{eq:BPcontact} and \eqref{eq:TMS_cond_asymptotics_1} are referred to as, respectively, the \emph{Bethe-Peierls contact condition} \cite{Bethe_Peierls-1935} and the \emph{Ter-Martyrosyan--Skornyakov condition} \cite{TMS-1956}, and express a boundary condition for the wave function in the vicinity of the origin, which is indeed the characteristic behaviour of the low-energy bound state for a Schr\"{o}dinger operator $-\Delta+V$ where $V$ has almost zero support and $s$-wave scattering length $a=-(4\pi\alpha)^{-1}$. Thus, $-\Delta_\alpha$ is recognised to be the \emph{Hamiltonian of point interaction in the $s$-wave channel, localised at $x=0$, and with inverse scattering length $\alpha$} in suitable units.

The spectrum of $-\Delta_\alpha$ is given by
\begin{equation}\label{eq:spectrumDalpha}
\begin{split}
\sigma_{\mathrm{ess}}(-\Delta_\alpha)\;&=\;\sigma_{\mathrm{ac}}(-\Delta_\alpha)\;=\;[0,+\infty)\,,\qquad \sigma_{\mathrm{sc}}(-\Delta_\alpha)\;=\;\emptyset\,, \\
\sigma_{\mathrm{p}}(-\Delta_\alpha)\;&=\;
\begin{cases}
\qquad \emptyset & \textrm{if }\alpha\in[0,+\infty] \\
\{-(4\pi\alpha)^2\} & \textrm{if }\alpha\in(-\infty,0)\,.
\end{cases}
\end{split}
\end{equation}
The negative eigenvalue $-(4\pi\alpha)^2$, when it exists, is simple and the corresponding eigenfunction is $|x|^{-1}e^{-4\pi|\alpha|\,|x|}$. Thus, $\alpha\geqslant 0$ corresponds to a non-confining, `repulsive' contact interaction.

We can now make \eqref{eq:sin_Hartree_formal} unambiguous and therefore consider the singular Hartree equation 
\begin{equation}\label{eq:sing_Hartree}
\ii\partial_t u\;=\;-\Delta_\alpha u+ (w*|u|^2)u\,.
\end{equation}
In order to avoid non-essential additional discussions, we restrict ourselves once and for all to positive $\alpha$'s. In fact, $-\Delta_\alpha$ is semi-bounded from below for every $\alpha\in(-\infty,+\infty]$, as seen in \eqref{eq:spectrumDalpha} above, thus shifting it up by a suitable constant one ends up with studying a modification of \eqref{eq:sing_Hartree} with a trivial linear term that does not affect the solution theory of the equation.

Owing to the self-adjointness of $-\Delta_\alpha$, and to its positivity for $\alpha\geqslant 0$, the `\emph{singular (or perturbed) Schr\"{o}dinger propagator}' $t\mapsto e^{\ii t\Delta_\alpha}$ leaves the domain of each power of $-\Delta_\alpha$ invariant. In complete analogy to the non-perturbed case, where the free Schr\"{o}dinger propagator $t\mapsto e^{\ii t\Delta}$ leaves the Sobolev space $H^s(\mathbb{R}^3)=\mathcal{D}((-\Delta)^{s/2})$ invariant, and the solution theory for the ordinary Hartree equation is made in $H^s(\mathbb{R}^3)$, including the energy space $H^1(\mathbb{R}^3)$, now the meaningful spaces of solutions where to settle the Cauchy problem for \eqref{eq:sing_Hartree} are of the type $\widetilde{H}^s_\alpha(\mathbb{R}^3)$, the `\emph{singular Sobolev space}' of order $s$, namely the Hilbert space 
\begin{equation}\label{eq:Hsperturbed}
\widetilde{H}^s_\alpha(\mathbb{R}^3)\;:=\;\mathcal{D}((-\Delta_\alpha)^{s/2})
\end{equation}
equipped with the `fractional singular Sobolev norm'
\begin{equation}\label{eq:Hsperturbed-norm}
\|\psi\|_{\widetilde{H}^s_\alpha}\;:=\;\|(\mathbbm{1}-\Delta_\alpha)^{s/2}\psi\|_2\,. 
\end{equation}

It is worth remarking that whereas the kernel of the propagator $t\mapsto e^{\ii t\Delta_\alpha}$ is known since long \cite{Scarlatti-Teta-1990,Albeverio_Brzesniak-Dabrowski-1995}, the characterisation of the singular fractional Sobolev space $\widetilde{H}^s_\alpha(\mathbb{R}^3)$ is only a recent achievement \cite{Georgiev-M-Scandone-2016-2017}, and we shall review it in Section \ref{sec:preparatory_material}.

In view of the preceding discussion, we consider the Cauchy problem
\begin{equation}\label{eq:Cauchy_problem_sing_Hartree}
\begin{cases}
\;\ii\partial_t u\;=\;-\Delta_\alpha u+ (w*|u|^2)u \\
\; u(0)\;=\;f\;\in\;\widetilde{H}^s_\alpha(\mathbb{R}^3)\,.
\end{cases}
\end{equation}
We are going to discuss its \emph{local} solution theory both in a regime of low (i.e., $s\in[0,\frac{1}{2}$), intermediate (i.e., $s\in(\frac{1}{2},\frac{3}{2})$), and high (i.e., $s\in(\frac{3}{2},2]$) regularity. Then, exploiting the conservation of the mass and the energy, we are going to obtain a \emph{global} theory in the mass space ($s=0$) and the energy space ($s=1$).

We deal with strong $\widetilde{H}^s_\alpha$-solutions of the problem \eqref{eq:Cauchy_problem_sing_Hartree}, meaning, functions $u\in\mathcal{C}(I,\widetilde{H}_\alpha^s(\mathbb{R}^3))$ for some interval $I\subseteq\mathbb{R}$ with $I\ni 0$, which are fixed points for the \emph{solution map}
\begin{equation}\label{eq:Smap}
\Phi(u)(t)\;:=\;e^{\ii t\Delta_\alpha} f-\ii\int_0^t e^{\ii (t-\tau)\Delta_\alpha}(w*|u(\tau)|^2)u(\tau)\,\ud\tau\,.
\end{equation}

Let us recall the notion of \emph{local} and \emph{global well-posedness} (see \cite[Section 3.1]{cazenave}).

\begin{definition}
	We say that the Cauchy problem \eqref{eq:Cauchy_problem_sing_Hartree} is locally well-posed in $\widetilde{H}_\alpha^s(\mathbb{R}^3)$ if the following properties hold:
	\begin{itemize}
		\item[(i)] For every $f\in\widetilde{H}_\alpha^s(\mathbb{R}^3)$, there exists a unique strong $\widetilde{H}_\alpha^s$-solution $u$ to the equation
		\begin{equation}\label{eq:integral_formula_Duhamel}
		u(t)\;=\;e^{\ii t\Delta_\alpha} f-\ii\int_0^t e^{\ii (t-\tau)\Delta_\alpha}(w*|u(\tau)|^2)u(\tau)\,\ud\tau
		\end{equation}
		defined on the maximal interval $(-T_*,T^*)$, where $T_*,T^*\in(0,+\infty]$ depend on $f$ only.
		\item[(ii)] There is the blow-up alternative: if $T^*<+\infty$ (resp., if $T_*<+\infty$), then $\lim_{t\uparrow T^*}\|u(t)\|_{\widetilde{H}_\alpha^s}=+\infty$ (resp., $\lim_{t\downarrow T_*}\|u(t)\|_{\widetilde{H}_\alpha^s}=+\infty$).
		\item[(iii)] There is continuous dependence on the initial data: if $f_n\xrightarrow[]{n\to+\infty} f$ in $\widetilde{H}_\alpha^s(\R^3)$, and if $I\subset(-T_*,T^*)$ is a closed interval, then the maximal solution $u_n$ to \eqref{eq:Cauchy_problem_sing_Hartree} with initial datum $f_n$ is defined on $I$ for $n$ large enough, and satisfies $u_n\xrightarrow[]{n\to+\infty}u$ in $\mathcal{C}(I,\widetilde{H}^s_\alpha(\R^3))$.
	\end{itemize}
	If $T_*=T^*=+\infty$, we say that the solution is global. If \eqref{eq:Cauchy_problem_sing_Hartree} is locally well-posed and for every $f\in\widetilde{H}_\alpha^s(\mathbb{R}^3)$ the solution is global, we say that \eqref{eq:Cauchy_problem_sing_Hartree} is globally well-posed in $\widetilde{H}_\alpha^s(\mathbb{R}^3)$.
\end{definition}

Let us emphasize the following feature of solutions to \eqref{eq:integral_formula_Duhamel}: if both $f$ and $w$ are spherically symmetric, so too is $u$. This follows at once from the symmetry  of the non-linear term of  \eqref{eq:integral_formula_Duhamel} together with the previously mentioned fundamental property that the subspaces of $L^2(\mathbb{R}^3)$ of definite rotational symmetry are invariant under the propagator $ e^{\ii t\Delta_\alpha}$. This makes the above definitions of strong solutions and well-posedness meaningful also with respect to the spaces
$$\widetilde{H}_{\alpha,\mathrm{rad}}^s(\R^3)\;:=\;\widetilde{H}_{\alpha}^s(\R^3)\cap L^2_{\ell=0}(\R^3)$$
equipped with the $\widetilde{H}_{\alpha}^s$-norm. Part of the solution theory we found is set in such spaces.

We can finally formulate our main results. Let us start with the local theory.

\begin{theorem}[$L^2$-theory -- local well-posedness]\label{thm:L2_reg}
	Let $\alpha\geqslant 0$. Let $w\in L^{\frac{3}{\gamma},\infty}(\mathbb{R}^3)$ for $\gamma\in[0,\frac32)$. Then the Cauchy problem \eqref{eq:Cauchy_problem_sing_Hartree} is locally well-posed in $L^2(\R^3)$.
\end{theorem}

\begin{theorem}[Low regularity -- local well-posedness]\label{thm:low_reg}
	Let $\alpha\geqslant 0$ and $s\in(0,\frac{1}{2})$. Let $w\in L^{\frac{3}{\gamma},\infty}(\mathbb{R}^3)$ for $\gamma\in[0,2s]$. Then the Cauchy problem \eqref{eq:Cauchy_problem_sing_Hartree} is locally well-posed in $\widetilde{H}_{\alpha}^s(\R^3)$, which in this regime coincides with $H^s(\R^3)$.
\end{theorem}

\begin{theorem}[Intermediate regularity -- local well-posedness]\label{thm:high_reg}
	Let $\alpha\geqslant 0$ and $s\in(\frac12,\frac{3}{2})$. Let $w\in W^{s,p}(\mathbb{R}^3)$ for $p\in(2,+\infty)$. Then the Cauchy problem \eqref{eq:Cauchy_problem_sing_Hartree} is locally well-posed in $\widetilde{H}_{\alpha}^s(\R^3)$.
\end{theorem}

\begin{theorem}[High regularity -- local well-posedness]\label{highh}
	Let $\alpha\geqslant 0$ and $s\in(\frac{3}{2},2]$. Let $w\in W^{s,p}(\mathbb{R}^3)$ for $p\in(2,+\infty)$ and spherically symmetric. Then the Cauchy problem \eqref{eq:Cauchy_problem_sing_Hartree} is locally well-posed in $\widetilde{H}_{\alpha,\mathrm{rad}}^s(\R^3)$.
\end{theorem}

The transition cases $s=\frac{1}{2}$ and $s=\frac{3}{2}$ are not covered explicitly for the mere reason that the structure of the perturbed Sobolev spaces $\widetilde{H}^{1/2}_\alpha(\mathbb{R}^3)$ and $\widetilde{H}^{3/2}_\alpha(\mathbb{R}^3)$ is not as clean as that of $\widetilde{H}^{s}_\alpha(\mathbb{R}^3)$ when $s\notin\{\frac{1}{2},\frac{3}{2}\}$ -- see Theorem \ref{thm:domain_s} below for the general case and Remark \ref{rem:transition_cases} for the peculiarities of the transition cases.

Let us remark that for $s>0$ we have an actual `continuity' in $s$ of the assumption on $w$ in the three Theorems \ref{thm:low_reg}, \ref{thm:high_reg}, and \ref{highh} above -- in the low regularity case our proof does not require any control on derivatives of $w$ and therefore we find it more informative to formulate the assumption in terms of the Lorentz space corresponding to $W^{s,p}(\mathbb{R}^3)$.

Such a `continuity' is due to the fact that under the hypotheses of Theorems \ref{thm:low_reg}, \ref{thm:high_reg}, and \ref{highh} we can work in a locally-Lipschitz regime of the non-linearity. When instead $s=0$ we have a `jump' in the form of an extra range of admissible potentials $w$, which is due to the fact that for the $L^2$-theory we are able to make use of the Strichartz estimates for the singular Laplacian.

Next, we investigate the global theory in the mass and in the energy spaces.

\begin{theorem}[Global solution theory in the mass space]\label{thm:GWP_Mass}
	Let $\alpha\geqslant 0$, and let $w\in L^{\infty}(\mathbb{R}^3)\cap W^{1,3}(\R^3)$, or  $w\in L^{\frac{3}{\gamma},\infty}(\mathbb{R}^3)$ for $\gamma\in(0,\frac32)$. Then the Cauchy problem \eqref{eq:Cauchy_problem_sing_Hartree} is globally well-posed in $L^2(\R^3)$.
\end{theorem}

\begin{theorem}[Global solution theory in the energy space]\label{thm:GWP}
	Let $\alpha\geqslant 0$, $w\in W^{1,p}_{\mathrm{rad}}(\mathbb{R}^3)$ for $p\in(2,+\infty)$, and $f\in\widetilde{H}^1_{\alpha,\mathrm{rad}}(\mathbb{R}^3)$.
	\begin{itemize}
		\item[(i)] There exists a constant $C_w>0$, depending only on $\|w\|_{W^{1,p}}$, such that if $\|f\|_{L^2}\leqslant C_w$, then the unique strong solution in $\widetilde{H}^1_{\alpha,\mathrm{rad}}(\mathbb{R}^3)$ to \eqref{eq:Cauchy_problem_sing_Hartree} with initial data $f$ is global.
		\item[(ii)] If $w\geqslant 0$, then the Cauchy problem \eqref{eq:Cauchy_problem_sing_Hartree} is globally well-posed in $\widetilde{H}^1_{\alpha,\mathrm{rad}}(\mathbb{R}^3)$.
	\end{itemize}
\end{theorem}

As stated in the Theorems above, part of the local and of the global solution theory is set for spherically symmetric potentials $w$ and solutions $u$. In a sense, this is the natural solution theory for the singular Hartree equation, for sufficiently high regularity. In particular, the spherical symmetry needed for the high regularity theory is induced naturally by the special structure of the space $\widetilde{H}^s_\alpha(\mathbb{R}^3)$ (as opposite to $H^s(\mathbb{R}^3)$, or also to $\widetilde{H}^s_\alpha(\mathbb{R}^3)$ for small $s$), where a boundary (`contact') condition holds between regular and singular component of $\widetilde{H}^s_\alpha$-functions.
In the concluding Section \ref{sec:spheric} we comment on this phenomenon, with the proofs of our main Theorems in retrospective.

Before concluding this general introduction, it is worth mentioning that the \emph{one-dimensional} version of the non-linear Schr\"{o}d\-ing\-er equation with point-like pseudo-potentials is much more deeply investigated and better understood, as compared to the so far virtually unexplored scenario in three dimensions.

On $L^2(\mathbb{R})$ the Hamiltonian of point interaction ``$-\frac{\ud^2}{\ud x^2}+\delta(x)$'' is constructed in complete analogy to $-\Delta_\alpha$, namely as a self-adjoint extension of $(-\frac{\ud^2}{\ud x^2})\big|_{C^\infty_0(\mathbb{R}\setminus\{0\})}$, which results in a larger variety (a two-parameter family) of realisations, each of which is qualified by an analogous boundary condition at $x=0$ \cite[Chapters I.3 and I.4]{albeverio-solvable}. In fact, such an analogy comes with a profound difference, for the one-dimensional Hamiltonians of point interaction are form-bounded perturbations of the Laplacian, unlike $-\Delta_\alpha$ with respect to $-\Delta$ on $L^2(\mathbb{R}^3)$, and hence much less singular and with a more easily controllable domain. For instance, among the other realisations,  one can non-ambiguously think of $-\frac{\ud^2}{\ud x^2}+\delta(x)$  as a form sum, $\delta(x)$ now denoting there the Dirac distribution.

In the last dozen years  a systematic analysis was carried out of the non-linear Schr\"{o}dinger equation in one dimension, mainly with local non-linearity, of the form
\[
\ii\partial_t u\;=\;-\big({\textstyle\frac{\ud^2}{\ud x^2}}+\delta(x)\big)u+\alpha|u|^{\gamma-1}u\,,
\]
or the analogous equation with $\delta'$-interaction instead of $\delta$-interaction, initially motivated by phenomenological models of short-range obstacles in non-linear transport \cite{Witthaut-Mossmann-Korsh-2005}. This includes local and global well-posedness in operator domain and energy space and blow-up phenomena  \cite{Adami-Sacchetti-2005-1D-NLS-delta,Adami_JPA2009_1D_NLS_with_delta,Adami-Noja-CMP2013-1D-NLS-deltaprime}, weak $L^p$-solutions \cite{Pava-Ferreira-2014_1D-NLS-delta-deltaprime}, scattering \cite{Banica-Visciglia-2016_Scattering-NLS-delta}, solitons \cite{Ianni-Coz-Royier-2017_NLS-delta,Ikeda-Takahisa-2017_global-1D-NLS-delta}, as well as more recent modifications of the non-linearity \cite{Ardila-2017_NLS-deltaprime_lognonlin}. None of such works has a three-dimensional counterpart.

\section{Preparatory materials}\label{sec:preparatory_material}

In this Section we collect an amount of materials available in the literature, which will be crucial for the following discussion.

We start with the following characterisation, proved by two of us in a recent collaboration with V.~Georgiev, of the fractional Sobolev spaces and norms naturally induced by $-\Delta_{\alpha}$, that is, the spaces $\widetilde{H}^s_\alpha(\mathbb{R}^3)$ introduced in \eqref{eq:Hsperturbed}-\eqref{eq:Hsperturbed-norm}.

\begin{theorem}[Perturbed Sobolev spaces and norms, \cite{Georgiev-M-Scandone-2016-2017}]\label{thm:domain_s}
	Let $\alpha\geqslant 0$, $\lambda>0$, and $s\in[0,2]$. The following holds.
	\begin{itemize}
		\item[(i)] If $s\in[0,\frac{1}{2})$, then
		\begin{equation}\label{eq:Ds_s0-1/2}
		\widetilde{H}^s_\alpha(\mathbb{R}^3)\;=\;H^s(\mathbb{R}^3)
		\end{equation}
		and 
		\begin{equation}\label{eq:equiv_of_norms_s012}
		\|\psi\|_{\widetilde{H}^s_\alpha}\;\approx\;\|\psi\|_{H^s}
		\end{equation}
		in the sense of equivalence of norms. The constant in \eqref{eq:equiv_of_norms_s012} is bounded, and bounded away from zero, uniformly in $\alpha$.
		\item[(ii)] If $s\in(\frac{1}{2},\frac{3}{2})$, then
		\begin{equation}\label{eq:Ds_s1/2-3/2}
		\widetilde{H}^s_\alpha(\mathbb{R}^3)\;=\;H^s(\mathbb{R}^3)\dotplus\mathrm{span}\{G_\lambda\}\,,
		\end{equation}
		where $G_\lambda$ is the function \eqref{eq:defGlambda}, and for arbitrary $\psi=\phi_\lambda+\kappa_\lambda\,G_\lambda\in \widetilde{H}^s_\alpha(\mathbb{R}^3)$
		\begin{equation}\label{eq:equiv_of_norms_s1232}
		\|\phi+\kappa_\lambda\,G_\lambda\|_{\widetilde{H}^s_\alpha}\;\approx\;\|\phi\|_{H^s}+(1+\alpha)|\kappa_\lambda|\,.
		\end{equation}
		\item[(iii)] If $s\in(\frac{3}{2},2]$, then
		\begin{equation}\label{eq:Ds_s3/2-2}
		\begin{split}
		\widetilde{H}^s_\alpha(\mathbb{R}^3)\;= \;\Big\{\psi\in L^2(\mathbb{R}^3)\,\Big|\,\psi=\phi_\lambda+\frac{\phi_\lambda(0)}{\,\alpha+\frac{\sqrt{\lambda}}{4\pi}\,}\,G_\lambda\textrm{ with }\phi_\lambda\in H^s(\mathbb{R}^3)\Big\}
		\end{split}
		\end{equation}
		and for arbitrary $\psi=\phi_\lambda+\frac{\phi_\lambda(0)}{\,\alpha+\frac{\sqrt{\lambda}}{4\pi}\,}\,G_\lambda\in \widetilde{H}^s_\alpha(\mathbb{R}^3)$
		\begin{equation}\label{eq:equiv_of_norms_s322}
		\big\|\phi_\lambda+{\textstyle\frac{\phi_\lambda(0)}{\,\alpha+\frac{\sqrt{\lambda}}{4\pi}\,}}\,G_\lambda\big\|_{\widetilde{H}^s_\alpha}\;\approx\;\|\phi_\lambda\|_{H^s}\,.
		\end{equation}
		The constant in \eqref{eq:equiv_of_norms_s322} is bounded, and bounded away from zero, uniformly in $\alpha$.
	\end{itemize}
\end{theorem}

\begin{remark}
	The case $s=0$ is trivial, the case $s=1$ reproduces the form domain of $-\Delta_{\alpha}$ given in \eqref{eq:op_dom-opaction} above, the case $s=2$ reproduces the operator domain \eqref{eq:form_dom-opaction}.
\end{remark}

\begin{remark}\label{rem:transition_cases}
	Separating the three regimes above, two different transitions occur (see \cite[Section 8]{Georgiev-M-Scandone-2016-2017}). When $s$ decreases from larger values, the first transition arises at $s=\frac{3}{2}$, namely the level of $H^s$-regularity at which continuity is lost. Correspondingly, the elements in $\widetilde{H}^{3/2}_\alpha(\mathbb{R}^3)$ still decompose into a regular $H^{\frac{3}{2}}$-part plus a multiple of $G_\lambda$ (singular part), and the decomposition is still of the form $\phi_\lambda+\kappa_\lambda\,G_\lambda$, except that now $\phi_\lambda$ cannot be arbitrary in $H^{\frac{3}{2}}(\mathbb{R}^3)$: indeed, $\phi_\lambda$ has additional properties, among which the fact that its Fourier transform is integrable (a fact that is false for generic $H^{\frac{3}{2}}$-functions), and for such $\phi_\lambda$'s the constant $\kappa_\lambda$ has a form that is completely analogous to the constant in \eqref{eq:Ds_s3/2-2}, that is,
	\[
	\kappa_\lambda\;=\;\frac{1}{\,\alpha+\frac{\sqrt{\lambda}}{4\pi}\,}\,\frac{1}{\,(2\pi)^{\frac{3}{2}}}\int_{\mathbb{R}^3}\!\ud p\,\widehat{\phi_\lambda}(p)
	\]
	(see \cite[Prop.~8.2]{Georgiev-M-Scandone-2016-2017}).
	Then, for $s<\frac{3}{2}$, the link between the two components disappears completely.
	Decreasing $s$ further, the next transition occurs at $s=\frac{1}{2}$, namely the level of $H^s$-regularity below which the Green's function itself belongs to $H^s(\mathbb{R}^3)$ and it does not necessarily carry the leading singularity any longer. At the transition $s=\frac{1}{2}$, the elements in $\widetilde{H}^{1/2}_\alpha(\mathbb{R}^3)$ still exhibit a decomposition into a regular $H^{\frac{1}{2}}$-part plus a more singular ${H^{\frac{1}{2}}}^-$-part, except that ${H^{\frac{1}{2}}}^-$-singularity is not explicitly expressed in terms of the Green's function $G_\lambda$ (see \cite[Prop.~8.1]{Georgiev-M-Scandone-2016-2017}). Then, for $s<\frac{1}{2}$, only $H^s$-functions form the fractional domain. Remarkably, yet in the same spirit, such transition thresholds $s=\frac{1}{2}$ and $s=\frac{3}{2}$ (and their analogues) emerge in several other contexts, such as the regularity of solutions to the cubic non-linear Schr\"{o}dinger equation on the half-line with Bourgain's restricted norm methods \cite{Erdogan-Tzirakis-2015_NLS-halfline} or the regime of rank-one singular perturbations of the fractional Laplacian \cite{MOS-2018_FractionalPerturbation,MS-2018-shrinking-and-fractional}.
\end{remark}

\begin{remark}
	In the limit $\alpha\to+\infty$ (recall that $\Delta_{\alpha=\infty}$ is the self-adjoint Laplacian on $L^2(\mathbb{R}^3)$) the equivalence of norms \eqref{eq:equiv_of_norms_s1232} tends to be lost, consistently with the fact that the function $G_\lambda$ does not belong to $H^s(\mathbb{R}^3)$. Instead, the norm equivalences \eqref{eq:equiv_of_norms_s012} and \eqref{eq:equiv_of_norms_s322} remain valid in the limit $\alpha\to+\infty$, which is also consistent with the structure of the space $\widetilde{H}^s_\alpha(\mathbb{R}^3)$ in those two cases.
\end{remark}

The second class of results we want to make use of were recently proved by two of us in collaboration with Dell'Antonio, Iandoli, and Yajima, and concern the dispersive properties of the propagator $t\mapsto e^{\ii t\Delta_\alpha}$ associated with $-\Delta_\alpha$, quantified both by (dispersive) pointwise-in-time estimates and by (Strichartz-like) space-time estimates.

To this aim, let us define a pair of exponents 
$(q,r)$ \emph{admissible for} $-\Delta_\alpha$ if 
\begin{equation}\label{eq:admie}
r\in[2,3)\qquad\textrm{and}\qquad 0\;\leqslant\;\frac{2}{q}\;=\;3\,\Big(\frac{1}{2}-\frac{1}{r}\Big)\;<\;\frac{1}{2}\,,
\end{equation}
that is, $q=\frac{4r}{3(r-2)}\in(4,+\infty]$.

\begin{theorem}[\cite{Iandoli-Scandone-2017,DMSY-2017}]\label{thm:Strichartz}~
	\begin{itemize}
		\item[(i)] There is a constant $C>0$ such that, for each $r\in[2,3)$,
		\begin{equation}\label{eq:p-q}
		\|e^{\ii t \Delta_\alpha}u\|_{L^r(\mathbb{R}^3)}\;\leqslant\;C\,
		|t|^{-3(\frac12-\frac1{r})}\|u\|_{L^{r'}(\mathbb{R}^3)}\,, \qquad  t\neq 0\,.
		\end{equation} 
		\item[(ii)] Let  $(q,r)$ and $(s,p)$ be two admissible pairs 
		for $-\Delta_\alpha$.
		and $r',s'$ are dual exponents of $r,s$. 
		Then, for a 
		constant $C>0$, 
		\begin{equation} \label{stri-1}
		\|e^{\ii t\Delta_\alpha}f\|_{L^q(\mathbb{R}_t, L^r(\mathbb{R}_x^3))}\;\leqslant\;
		C\|f\|_{L^2(\mathbb{R}^3)}
		\end{equation} 
		and  
		\begin{equation} \label{stri-2}
		\left\| 
		\int_0^t e^{\ii(t-\tau)\Delta_\alpha}F(\tau)\,\ud \tau
		\right\|_{L^q(\mathbb{R}_t, L^r(\mathbb{R}^3_x))}
		\;\leqslant\;
		C\|F\|_{L^{s'}(\mathbb{R}_t, L^{p'}(\mathbb{R}^3_x))}\,.
		\end{equation}
	\end{itemize}
\end{theorem}

\begin{remark}
	The dispersive estimate \eqref{eq:p-q} has a precursor in \cite{Dancona-Pierfelice-Teta-2006} in the form of the weighted $L^1_tL_x^\infty$-estimate
	\begin{equation}\label{eq:p-q_weighted}
	\|w^{-1}e^{\ii t \Delta_\alpha}u\|_{L^\infty(\mathbb{R}^3)}\;\leqslant\;C\,
	|t|^{-\frac{3}{2}}\|w\,u\|_{L^{1}(\mathbb{R}^3)}\,, \qquad  t\neq 0\,,
	\end{equation} 
	where $w(x):=1+|x|^{-1}$. The weight $w$ is needed to compensate the $|x|^{-1}$ singularity naturally emerging in $e^{\ii t \Delta_\alpha}u$ for any $t\neq 0$, as typical for a  generic element of the energy space $\widetilde{H}^s_\alpha(\mathbb{R}^3)$ -- see \eqref{eq:Ds_s1/2-3/2} above. By interpolation between \eqref{eq:p-q_weighted} and $\|e^{\ii t\Delta_\alpha}u\|_2=\|u\|_2$ one can then obtain a weighted version of \eqref{eq:p-q} in the whole regime $r\in[2,+\infty]$ (see \cite[Prop.~4]{Iandoli-Scandone-2017} or \cite[Eq.~(1.19)]{DMSY-2017}. It was a merit of \cite{Iandoli-Scandone-2017} to have  observed that as long as $r\in[2,3)$, the dispersive estimate \eqref{eq:p-q} is valid also without weights. In parallel, in \cite{DMSY-2017} the same estimate \eqref{eq:p-q} was obtained as a corollary of the much stronger result of the $L^r$-boundedness, $r\in(1,3)$, of the wave operators
	\[
	W_\alpha^{\pm}\;=\;\mathrm{strong}\;\textrm{-}\!\!\!\lim_{t\to\pm\infty}e^{-\ii t \Delta_\alpha}e^{\ii t\Delta}
	\]
	associated to the pair $(-\Delta_\alpha,-\Delta)$, and of the intertwining properties of $W^{\pm}$, which allow one to deduce \eqref{eq:p-q} in the regime $r\in[2,3)$ from the analogous and well-known dispersive estimate for the free propagator $t\mapsto e^{\ii t\Delta}$.
\end{remark}

The third class of properties that we need to recall concern fundamental tools of fractional calculus. One is the following fractional Leibniz rule by Kato and Ponce, also in the generalised version by Gulisashvili and Kon.

\begin{theorem}[Generalised fractional Leibniz rule, \cite{Kato-Ponce_CommEst-1988,Gulisashvili-Kon-1996}]\label{KatoPonce}
	Suppose that $r\in(1,+\infty)$ and $p_1,p_2,q_1,q_2\in(1,+\infty]$ with $\frac{1}{p_j}+\frac{1}{q_j}=\frac{1}{r}$, $j\in\{1,2\}$, and suppose that $s,\mu,\nu\in[0,+\infty)$. Let $d\in\mathbb{N}$, then
	\begin{equation}\label{eq:KatoPonce}
	\begin{split}
	\|\mathcal{D}^s(fg)\|_{L^r(\mathbb{R}^d)}\;&\lesssim\;\|\mathcal{D}^{s+\mu}f\|_{L^{p_1}(\mathbb{R}^d)}\|\mathcal{D}^{-\mu}g\|_{L^{q_1}(\mathbb{R}^d)} \\
	&\qquad+\|\mathcal{D}^{-\nu}f\|_{L^{p_2}(\mathbb{R}^d)}\|\mathcal{D}^{s+\nu}g\|_{L^{q_2}(\mathbb{R}^d)}\,,
	\end{split}
	\end{equation}
	where $\mathcal{D}^s=(-\Delta)^{\frac{s}{2}}$, the Riesz potential. The same result holds when $\mathcal{D}^s$ is the Bessel potential $(\mathbbm{1}-\Delta)^{\frac{s}{2}}$.
\end{theorem}

\begin{remark}
	As a direct consequence of Mihlin multiplier theorem \cite[Section 6.1]{Bergh-Lofsrom_InterpolationSpaces1976}, the estimate \eqref{eq:KatoPonce} holds as well for $\mathcal{D}^s=(-\Delta+\lambda\mathbbm{1})^{\frac{s}{2}}$ for any $\lambda\geqslant 0$.
\end{remark}

We also need a more versatile re-distribution of the derivatives among the two factors $f$ and $g$ in \eqref{eq:KatoPonce}: the following recent result by Fujiwara, Georgiev, and Ozawa provides a very useful refinement of the fractional Leibniz rule and is based on a careful treatment of the correction term 
\begin{equation}
[f,g]_{s}\;:=\;f\,\mathcal{D}^sg+g\,\mathcal{D}^sf\,.
\end{equation}

\begin{theorem}[Higher order fractional Leibniz rule, \cite{Fujiwara-Georgiev-Ozawa-2016}]\label{KatoPonce_Vladimir}
	Suppose that $p,q,r\in(1,+\infty)$ with $\frac{1}{p}+\frac{1}{q}=\frac{1}{r}$ and let $d\in\mathbb{N}$.
	\begin{itemize}
		\item[(i)] Let $s_1,s_2\in[0,1]$ and set $s:=s_1+s_2$. Then
		\begin{equation}\label{eq:KatoPonce_Vladimir}
		\|\mathcal{D}^s(fg)-[f,g]_{s}\|_{L^r(\mathbb{R}^d)}\;\lesssim\;\|\mathcal{D}^{s_1}f\|_{L^{p}(\mathbb{R}^d)}\|\mathcal{D}^{s_2}g\|_{L^{q}(\mathbb{R}^d)}\,.
		\end{equation}
		\item[(ii)] Let $s_1\in[0,2]$, $s_2\in[0,1]$ be such that $s:=s_1+s_2\geqslant 1$. Then
		\begin{eqnarray}
		& & \|\mathcal{D}^s(fg)-[f,g]_{s}+s\,\mathcal{D}^{s-2}(\nabla f\cdot\nabla g)+s\mathcal{D}^{s-2}(g\Delta f)-sg\mathcal{D}^{s-2}\Delta f\|_{L^r(\mathbb{R}^d)} \nonumber \\
		& & \qquad\qquad\qquad\lesssim\;\|\mathcal{D}^{s_1}f\|_{L^{p}(\mathbb{R}^d)}\|\mathcal{D}^{s_2}g\|_{L^{q}(\mathbb{R}^d)}\,. \label{eq:KatoPonce_Vladimir_tre}
		\end{eqnarray}
		Moreover, since 
		\[
		\mathcal{D}^{s-2}(\nabla f\cdot\nabla g)+\mathcal{D}^{s-2}(g\Delta f)-g\mathcal{D}^{s-2}\Delta f=\mathcal{D}^{s-2}\nabla\cdot(g\nabla f)+g\mathcal{D}^sf
		\]
		we can rewrite \eqref{eq:KatoPonce_Vladimir_tre} in the more compact form
		\begin{equation}\label{eq:KatoPonce_Vladimir_quattro}
		\begin{split}
		\|\mathcal{D}^s(fg)-f\mathcal{D}^s&g+(s-1)g\mathcal{D}^sf+s\,\mathcal{D}^{s-2}\nabla\cdot(g\nabla f)\|_{L^r(\mathbb{R}^d)}\\
		&\lesssim\;\|\mathcal{D}^{s_1}f\|_{L^{p}(\mathbb{R}^d)}\|\mathcal{D}^{s_2}g\|_{L^{q}(\mathbb{R}^d)}\,.
		\end{split}
		\end{equation}
	\end{itemize}
\end{theorem}

For the fractional derivative of $|x|^{-1}e^{-\lambda|x|}$ we need, additionally, a point-wise estimate.

\begin{lemma}\label{lederi}
	Let $\lambda>0$, $s\in (0,2]$. We have the estimate
	\begin{equation}\label{deripower}
	\Big|\mathcal{D}^s\,\frac{e^{-\lambda|x|}}{|x|}\Big|\;\lesssim \;\frac{e^{-\lambda|x|}}{|x|}+\frac{e^{-\lambda|x|}}{|x|^{1+s}}\,,\qquad x\neq 0\,.
	\end{equation}
\end{lemma}

\begin{proof}
	Obvious when $s=2$, and a straightforward consequence of the identity
	\[
	\begin{split}
	\mathcal{D}^sG_\lambda\;&=\;\mathcal{F}^{-1}(|p|^s\widehat{G}_\lambda(p)) \;=\;-\mathcal{F}^{-1}\!\Big(\frac{1}{\;(2\pi)^{\frac{3}{2}}}\,\frac{\lambda}{p^2+\lambda} \Big)+\mathcal{F}^{-1}\!\Big(\frac{1}{\;(2\pi)^{\frac{3}{2}}}\,\frac{|p|^s+\lambda}{p^2+\lambda} \Big)
	\end{split}
	\]
	when $s\in(0,2)$.
\end{proof}

In the second part of this Section, based on the preceding properties, we derive two useful estimates that we are going to apply systematically in our discussion when $s>\frac{1}{2}$. Let us recall that in this case 
\begin{equation}\label{eq:algebra}
\|g_1 g_2\|_{W^{s,q}}\;\lesssim\;\|g_1\|_{W^{s,q}}\,\|g_2\|_{W^{s,q}}\qquad \qquad\quad (s>\textstyle{\frac{1}{2}},\;q\in(6,+\infty))\,,
\end{equation}
as follows, for example, from the fractional Leibniz rule \eqref{eq:KatoPonce} and Sobolev's embedding $W^{s,q}(\mathbb{R}^3)\hookrightarrow L^\infty(\mathbb{R}^3)$.

We start with the estimate for the regime $s\in(\frac{1}{2},\frac{3}{2})$, for which we recall Sobolev's embedding
\begin{equation}\label{eq:SobEmb-I}
\begin{split}
W^{s,q}(\R^3)\;&\hookrightarrow\; \mathcal{C}^{0,\vartheta(s,q)}(\R^3) \\
\vartheta(s,q)\;&:=\;\textstyle\min\{s-\frac{3}{q},1\}\qquad\qquad\qquad (s\in(\frac{1}{2},\frac{3}{2}),\;q\in(6,+\infty)\,.
\end{split}
\end{equation}

\begin{proposition}\label{tst}
	Let $s\in(\frac12,\frac32)$ and $h\in W^{s,q}(\R^3)$ for some $q\in(6,+\infty)$. Then
	\begin{equation}\label{tecre}
	\big\|\mathcal{D}^s\big((h-h(0))\,G_\lambda\big)\big\|_{L^2}\;\lesssim \;\|h\|_{W^{s,q}(\R^3)}\,,
	\end{equation}
	where $G_\lambda$ is the function \eqref{eq:defGlambda} for some $\lambda>0$.
\end{proposition}

\begin{proof}
	It is not restrictive to fix $\lambda=1$. For short, we set $G(x):=|x|^{-1}e^{-|x|}$ and $\widetilde{G}(x):=|x|^{-1}e^{-\frac{1}{2}|x|}$.

	By means of the commutator bound \eqref{eq:KatoPonce_Vladimir} we find
	\begin{equation*}\tag{i}
	\begin{split}
	\big\|\mathcal{D}^s\big((h&-h(0))\,G\big)\big\|_{L^2}\;=\;\big\|\mathcal{D}^s\big(e^{-\frac{1}{2}|x|}(h-h(0))\,\widetilde{G}\big)\big\|_{L^2} \\
	&\leqslant\; \big\|\mathcal{D}^s\big(e^{-\frac{1}{2}|x|}(h-h(0))\,\widetilde{G}\big) - \big[e^{-\frac{1}{2}|x|}(h-h(0)),\widetilde{G} \big]_s\big\|_{L^2}  \\
	&\qquad +\big\|e^{-\frac{1}{2}|x|}(h-h(0))\,\mathcal{D}^s\widetilde{G}\big\|_{L^2}+\big\|\widetilde{G}\,\mathcal{D}^s\big(e^{-\frac{1}{2}|x|}(h-h(0))\big)\big\|_{L^2}\\
	&\lesssim\; \big\|\mathcal{D}^{s_1}\big(e^{-\frac{1}{2}|x|}(h-h(0))\big)\big\|_{L^{q_1}}\|\mathcal{D}^{s_2}\widetilde{G}\|_{L^{q_2}}\\
	&\qquad +\big\|e^{-\frac{1}{2}|x|}(h-h(0))\,\mathcal{D}^s\widetilde{G}\big\|_{L^2}+\big\|\widetilde{G}\,\mathcal{D}^s\big(e^{-\frac{1}{2}|x|}(h-h(0))\big)\big\|_{L^2}\\
	&\equiv\;\mathcal{R}_1+\mathcal{R}_2+\mathcal{R}_3
	\end{split}
	\end{equation*}
	for every $s_1,s_2\in[0,1]$ with $s_1+s_2=s$ and every $q_1,q_2\in[2,+\infty]$ such that $q_1^{-1}+q_2^{-1}=2^{-1}$.

	Let us estimate the term $\mathcal{R}_3$. Since, by \eqref{eq:algebra} and Sobolev's embedding,
	\begin{equation*}\tag{ii}
	\begin{split}
	\big\|\mathcal{D}^s\big(e^{-\frac{1}{2}|x|}(h-h(0))\big)\big\|_{L^q}\;&\lesssim\;\big\|(\mathbbm{1}-\Delta)^{\frac{s}{2}}\big(e^{-\frac{1}{2}|x|}(h-h(0))\big)\big\|_{L^q} \\
	&\leqslant\;\big\|e^{-\frac{1}{2}|x|}h\big\|_{W^{s,q}}+|h(0)|\;\|e^{-\frac{1}{2}|x|}\|_{W^{s,q}} \\
	&\lesssim\;\|h\|_{W^{s,q}}+\|h\|_{L^{\infty}}\;\lesssim\; \|h\|_{W^{s,q}}\,,                                                                                                                                                                    
	\end{split}
	\end{equation*}
	and since $\|\widetilde{G}\|_{\frac{2q}{q-2}}<+\infty$ because $\frac{2q}{q-2}<3$ for $q>6$, then Holder's inequality yields
	\begin{equation*}\tag{iii}
	\mathcal{R}_3\;\leqslant\; \big\|\mathcal{D}^s\big(e^{-\frac{1}{2}|x|}(h-h(0))\big)\big\|_{L^q}\;\big\|\widetilde{G}\big\|_{L^{\frac{2q}{q-2}}}\;\lesssim\;\|h\|_{W^{s,q}}\,.
	\end{equation*}

	Next, let us estimate $\mathcal{R}_1$. When $s\in(\frac{1}{2},1]$, we choose $s_1=s$, $s_2=0$, $q_1=q$, and $q_2=\frac{2q}{q-2}$ and we proceed exactly as for $\mathcal{R}_3$. When instead $s\in(1,\frac32)$, we choose $s_1=1$, $s_2=s-1\in(0,\frac 12)$, $q_1=(\frac{1}{q}-\frac{s-1}{3})^{-1}$, and $q_2=(\frac{1}{2}-\frac{1}{q_1})^{-1}$. Then Sobolev's embedding $W^{s,q}(\mathbb{R}^3)\hookrightarrow W^{1,q_1}(\mathbb{R}^3)$ and estimate (ii) above imply
	$$\Big\|\mathcal{D}^{s_1}\big(e^{-\frac{1}{2}|x|}(h-h(0))\big)\Big\|_{L^{q_1}}\;\lesssim\; \Big\|(\mathbbm{1}-\Delta)^{\frac{s}{2}}\big(e^{-\frac{1}{2}|x|}(h-h(0))\big)\Big\|_{L^{q}}\;\lesssim\; \|h\|_{W^{s,q}}\,,$$
	whereas estimate \eqref{deripower} and the fact that $q_2<sq_2<3$ for $s\in(1,\frac32)$ imply
	\[
	\|\mathcal{D}^{s_2}\widetilde{G}\|_{L^{q_2}}\;\lesssim\;\big\|e^{-\frac{1}{2}|x|}\big({\textstyle\frac{1}{|x|}+\frac{1}{|x|^s}}\big) \big\|_{L^{q_2}}\;<\;+\infty\,.
	\]
	Thus, in either case $s\in\big(\frac12,1\big]$ and $s\in\big(1,\frac32\big)$,
	\begin{equation*}\tag{iv}
	\mathcal{R}_1\;\lesssim\;\|h\|_{W^{s,q}}\,.
	\end{equation*}

	Last, let us estimate $\mathcal{R}_2$. Because of the embedding \eqref{eq:SobEmb-I},
	\[
	\Big\|\frac{e^{-\frac{1}{2}|x|}(h-h(0))}{|x|^{\vartheta(s,q)}}\Big\|_{L^{\infty}}\;\lesssim\;\|h\|_{W^{s,q}}\,.
	\]
	Moreover, since $\vartheta(s,q)>s-\frac12$ and hence $2(1-\vartheta(s,q))<2(1+s-\vartheta(s,q))<3$ for every $s\in\big(\frac12,\frac32\big)$, estimate \eqref{deripower} implies
	$$\big\||x|^{\vartheta(s,q)}\mathcal{D}^s\widetilde{G}\big\|_{L^2}\;\lesssim\;\big\|e^{-\frac{1}{2}|x|}\big(\,|x|^{-(1-\vartheta(s,q))}+|x|^{-(1+s-\vartheta(s,q))}\big) \big\|_{L^2}
	\;<\;+\infty\,.$$
	Thus,
	\begin{equation*}\tag{v}
	\mathcal{R}_2\;\leqslant\;\Big\|\frac{e^{-\frac{1}{2}|x|}(h-h(0))}{|x|^{\vartheta(s,q)}}\Big\|_{L^{\infty}}\,\big\||x|^{\vartheta(s,q)}\mathcal{D}^s\widetilde{G}\big\|_{L^2}\;\lesssim\;\|h\|_{W^{s,q}}\,.
	\end{equation*}
	Plugging (iii), (iv), and (v) into (i) the thesis follows.
\end{proof}

We establish now an analogous estimate for the regime $s\in(\frac32,2]$, for which we recall Sobolev's embedding
\begin{equation}\label{eq:SobEmb-II}
\begin{split}
W^{s,q}(\R^3)\;&\hookrightarrow\; \mathcal{C}^{1,\vartheta(s,q)}(\R^3) \\
\vartheta(s,q)\;&:=\;\textstyle s-1-\frac3q\qquad\quad (s\in(\frac{3}{2},2],\;q\in(6,+\infty))\,.
\end{split}
\end{equation}

\begin{proposition}\label{tstdue}
	Let $s\in(\frac32,2]$ and $h\in W^{s,q}(\R^3)$ for some $q\in(6,+\infty)$. 
	Assume further that $h$ is spherically symmetric and that $(\nabla h)(0)=0$.
	Then
	\begin{equation}\label{tecredue}
	\big\|\mathcal{D}^s\big((h-h(0))\,G_\lambda\big)\big\|_{L^2}\;\lesssim \;\|h\|_{W^{s,q}(\R^3)}\,,
	\end{equation}
	where $G_\lambda$ is the function \eqref{eq:defGlambda} for some $\lambda>0$.
\end{proposition}

Prior to proving Proposition \ref{tstdue} let us highlight the following property.

\begin{lemma}\label{le:fra_decay}
	Under the assumptions of Proposition \ref{tstdue}, 
	\begin{equation}\label{eq:fra_decay}
	|h(x)-h(0)|\;\lesssim \;\|h\|_{W^{s,q}}\;|x|^{1+\vartheta(s,q)}\,,
	\end{equation}
	where $\vartheta(s,q)=s-1-\frac3q$, as fixed in \eqref{eq:SobEmb-II}.
\end{lemma}

\begin{proof}
	By assumption, $h(x)=\widetilde{h}(|x|)$ for some even function $\widetilde{h}:\R\rightarrow\C$. Owing to the embedding \eqref{eq:SobEmb-II}, $\widetilde{h}\in\mathcal{C}^{1,\vartheta(s,q)}(\R)$, whence  $\widetilde{h}'\in\mathcal{C}^{0,\vartheta(s,q)}(\R)$. Moreover, $\widetilde{h}'(0)=0$, because $(\nabla h)(0)=0$. Therefore,
	\begin{equation*}
	|\widetilde{h}'(\rho)|\;=\;|\widetilde{h}'(\rho)-\widetilde{h}'(0)|\;\lesssim \;\|h\|_{W^{s,q}}\:|\rho|^{\vartheta(s,q)}\,. 
	\end{equation*}
	As a consequence,
	\begin{equation*}
	|h(x)-h(0)|\;\leqslant\;\int_0^{|x|}|\widetilde{h}'(\rho)|\,\ud\rho\;\lesssim\; \|h\|_{W^{s,q}}\!\int_0^{|x|}\rho^{\vartheta(s,q)}\,\ud \vartheta\;\lesssim\;\|h\|_{W^{s,q}}\:|x|^{1+\vartheta(s,q)}\,,
	\end{equation*}
	which completes the proof.
\end{proof}

\begin{proof}[Proof of Proposition \ref{tstdue}]
	It is not restrictive to fix $\lambda=1$. For short, we set $G(x):=|x|^{-1}e^{-|x|}$ and $\widetilde{G}(x):=|x|^{-1}e^{-\frac{1}{2}|x|}$.

	Let us split
	\begin{align}
	\big\|&\mathcal{D}^s\big((h-h(0))\,G\big)\big\|_{L^2}\;=\;\big\|\mathcal{D}^s\big(e^{-\frac{1}{2}|x|}(h-h(0))\,\widetilde{G}\big)\big\|_{L^2} \nonumber \\
	&\leqslant\; \big\|\mathcal{D}^s\big(e^{-\frac{1}{2}|x|}(h-h(0))\,\widetilde{G}\big) - e^{-\frac{1}{2}|x|}(h-h(0))\mathcal{D}^s\widetilde{G} \nonumber \\
	&\quad+(s-1)\widetilde{G}\,\mathcal{D}^s\big(e^{-\frac{1}{2}|x|}(h-h(0))\big)+s\,\mathcal{D}^{s-2}\nabla\cdot\big(\widetilde{G}\,\nabla\big(e^{-\frac{1}{2}|x|}(h-h(0))\big)\big\|_{L^2} \nonumber \\
	&\quad +\big\|e^{-\frac{1}{2}|x|}(h-h(0))\,\mathcal{D}^s\widetilde{G}\big\|_{L^2}+(s-1)\big\|\widetilde{G}\,\mathcal{D}^s\big(e^{-\frac{1}{2}|x|}(h-h(0))\big)\big\|_{L^2} \nonumber \\
	&\quad +s\,\big\|\mathcal{D}^{s-2}\nabla\cdot\big(\widetilde{G}\nabla\big(e^{-\frac{1}{2}|x|}(h-h(0))\big)\big\|_{L^2} \nonumber \\
	&\equiv\;\mathcal{R}_1+\mathcal{R}_2+\mathcal{R}_3+\mathcal{R}_4\,. \tag{i}
	\end{align}

	We estimate the term $\mathcal{R}_1$ by means of the commutator bound \eqref{eq:KatoPonce_Vladimir_quattro} with 
	$s_1=s$ and $s_2=0$ and of \eqref{eq:algebra}, namely
	\begin{equation*}\tag{ii}
	\mathcal{R}_1\;\lesssim \;\Big\|\mathcal{D}^s\big(e^{-\frac{1}{2}|x|}(h-h(0))\big)\Big\|_{L^q}\|\widetilde{G}\|_{L^{\frac{2q}{q-2}}}\;\lesssim\; \|h\|_{W^{s,q}}
	\end{equation*}
	(since $\frac{2q}{q-2}\in[2,3)$ and $\|\widetilde{G}\|_{\frac{2q}{q-2}}<+\infty$).

	For the estimate of $\mathcal{R}_2$, we observe that $s-\vartheta(s,q)<\frac{3}{2}$ and hence \eqref{deripower} implies
	\[
	\big\|\,|x|^{1+\vartheta(s,q)}\,\mathcal{D}^s\widetilde{G}\big\|_{L^2}\;\lesssim\;\big\|\,e^{-\frac{1}{2}|x|}\big(\,|x|^{\vartheta(s,q)}+|x|^{-(s-\vartheta(s,q))}\big)\big\|_{L^2}\;<\;+\infty\,;
	\]
	this and the bound \eqref{eq:fra_decay} yield
	\begin{equation*}\tag{iii}
	\mathcal{R}_2\;\leqslant\;\Big\|\frac{h-h(0)}{\;|x|^{1+\vartheta(s,q)}}\Big\|_{L^{\infty}}\;\big\|\,|x|^{1+\vartheta(s,q)}\mathcal{D}^s\widetilde{G}\big\|_{L^2}\;\lesssim\; \|h\|_{W^{s,q}}\,.
	\end{equation*}

	For $\mathcal{R}_3$, H\"older's inequality, the property \eqref{eq:algebra}, and Sobolev's embedding yield
	\begin{equation*}\tag{iv}
	\begin{split}
	\mathcal{R}_3\;&\lesssim\; \big\|\mathcal{D}^s\big(e^{-\frac{1}{2}|x|}(h-h(0))\big)\|_{L^q}\,\|\widetilde{G}\|_{L^{\frac{2q}{q-2}}} \\
	&\lesssim\;\big\|e^{-\frac{1}{2}|x|}(h-h(0))\|_{W^{s,q}}\;\lesssim\;\|h\|_{W^{s,q}}+\|h\|_{L^\infty} \\
	&\lesssim\; \|h\|_{W^{s,q}}.
	\end{split}
	\end{equation*}

	For $\mathcal{R}_4$, one has
	\begin{equation*}\tag{v}
	\begin{split}
	\mathcal{R}_4\;&\lesssim\; \big\|\mathcal{D}^{s-1}\big(\widetilde{G}\,\nabla\big(e^{-\frac{1}{2}|x|}(h-h(0)\big)\big)\big\|_{L^2}\\
	&\lesssim\;\big\|\nabla\big(e^{-\frac{1}{2}|x|}(h-h(0)\big)\big\|_{W^{s-1,q}}\;\lesssim \;\| h \|_{W^{s,q}}\,,
	\end{split}
	\end{equation*}
	where we used the estimate \eqref{tecre} in the second inequality  (indeed, $s-1\in(\frac{1}{2},1)$), and the property \eqref{eq:algebra} and Sobolev's embedding in the last inequality. 
	
	Plugging the bounds (ii)-(v) into (i) completes the proof.
\end{proof}

\section{$L^2$-theory and low regularity theory}

In this Section we prove Theorems \ref{thm:L2_reg} and \ref{thm:low_reg}. Let us start with Theorem \ref{thm:low_reg} and then discuss the adaptation for $s=0$. The proof for $s\in(0,\frac{1}{2})$ is based on a fixed point argument in the complete metric space $(\mathcal{X}_{T,M},d)$ defined by
\begin{equation}\label{eq:fixed_point_space}
\begin{split}
\mathcal{X}_{T,M}\;&:=\;\big\{u\in L^\infty([-T,T],\widetilde{H}_\alpha^s(\mathbb{R}^3))\,|\,\|u\|_{L^\infty([-T,T],\widetilde{H}_\alpha^s(\mathbb{R}^3)}\leqslant M\big\} \\
d(u,v)\;&:=\;\|u-v\|_{L^\infty([-T,T],L^2(\mathbb{R}^3))}
\end{split}
\end{equation}
for given $T,M>0$. This is going to be the same space for the contraction argument in the intermediate regularity regime $s\in(\frac{1}{2},\frac{3}{2})$ (Section \ref{sec:intermediate}), whereas for the high regularity regime $s\in(\frac{3}{2},2)$ (Section \ref{sec:higher_reg}) we are going to only use the spherically symmetric sector of the space \eqref{eq:fixed_point_space}.

\begin{proof}[Proof of Theorem  \ref{thm:low_reg}]~

	Since by assumption $s\in(0,\frac{1}{2})$, the spaces $H^s(\mathbb{R}^3)$ and $\widetilde{H}^s_\alpha(\mathbb{R}^3)$ coincide and their norms are equivalent (Theorem \ref{thm:domain_s}), so we can interchange them in the computations that follow.

	From the expression \eqref{eq:Smap} for the solution map $\Phi(u)$ one finds
	\[
	\|\Phi(u)\|_{L^\infty\widetilde{H}_\alpha^s }\;\leqslant\;\|f\|_{\widetilde{H}_\alpha^s }+T\,\|(w*|u|^2)u\|_{L^\infty\widetilde{H}_\alpha^s }
	\]
	and applying the fractional Leibniz rule \eqref{eq:KatoPonce} (Theorem \ref{KatoPonce}), H\"{o}lder's inequality, and Young's inequality one also finds
	\[
	\begin{split}
	\|(w*|u|^2)&u\|_{L^\infty\widetilde{H}_\alpha^s}\;=\;\|(w*|u|^2)u\|_{L^\infty H^s } \\
	&\lesssim\;\|w*|u|^2\|_{L^\infty L^\infty }\,\|\mathcal{D}^su\|_{L^\infty L^2 } \\
	&\qquad +\|\mathcal{D}^s(w*|u|^2)\|_{L^\infty L^{\frac{6}{\gamma},\infty} }\,\|u\|_{L^\infty L^{\frac{6}{3-\gamma}} } \\
	&\lesssim\;\|w\|_{L^{\frac{3}{\gamma},\infty}\!(\mathbb{R}^3)}   \, \|u\|^2_{L^\infty L^{\frac{6}{3-\gamma}} }\,\|\mathcal{D}^su\|_{L^\infty L^2 }\,. \\
	\end{split}
	\]
	Sobolev's embedding $\widetilde{H}_\alpha(\mathbb{R}^3)=H^s(\mathbb{R}^3)\hookrightarrow H^{\frac{\gamma}{2}}(\mathbb{R}^3)\hookrightarrow L^{\frac{6}{3-\gamma}}(\mathbb{R}^3)$ then yields
	\[\tag{i}
	\|\Phi(u)\|_{L^\infty \widetilde{H}_\alpha^s(\mathbb{R}^3)}\;\leqslant\;\|f\|_{ \widetilde{H}_\alpha^s}+C_1\,T\,\|w\|_{L^{\frac{3}{\gamma},\infty}} \,\|u\|^3_{L^\infty \widetilde{H}_\alpha^s }\,.
	\]
	for some constant $C_1>0$.

	On the other hand, again by H\"{o}lder's and Young's inequality,
	\[
	\begin{split}
	\|\Phi(u)-\Phi(v)&\|_{L^\infty L^2 }\;\leqslant\;T\,\|(w*|u|^2)u-(w*|v|^2)v\|_{L^\infty L^2 } \\
	&\lesssim\;T\,\Big( \|(w*|u|^2)(u-v)\|_{L^\infty L^2 } + \big\|\big(w*(|u|^2-|v|^2\big) v\big\|_{L^\infty L^2 }\Big) \\
	&\lesssim\;T\,\Big( \|w\|_{L^{\frac{3}{\gamma},\infty}}   \, \|u\|^2_{L^\infty L^{\frac{6}{3-\gamma}} }\,\|u-v\|_{L^\infty L^2} \\
	&\qquad\qquad + \|w\|_{L^{\frac{3}{\gamma},\infty}}   \,\|u-v\|_{L^\infty L^2 }\,\||u|+|v|\|_{L^\infty L^{\frac{6}{3-\gamma}} }\|v|\|_{L^\infty L^{\frac{6}{3-\gamma}} }\,\Big)\,,
	\end{split}
	\]
	whence, by the same embedding $\widetilde{H}_\alpha(\mathbb{R}^3)\hookrightarrow L^{\frac{6}{3-\gamma}}(\mathbb{R}^3)$ as before,
	\[\tag{ii}
	\begin{split}
	d(\Phi(u),\Phi(v))\;&\leqslant\;C_2\,\|w\|_{L^{\frac{3}{\gamma},\infty}}\big(\|u\|^2_{L^\infty \widetilde{H}_\alpha^s }+\|v\|^2_{L^\infty \widetilde{H}_\alpha^s } \big)\: T\,d(u,v)
	\end{split}
	\]
	for some constant $C_2>0$.

	Thus, choosing $T$ and $M$ such that
	\[
	M\;=\;2\,\|f\|_{\widetilde{H}_\alpha^s}\,,\qquad T\;=\;{\textstyle\frac{1}{4}}\,\big(\max\{C_1,C_2\}\,M^2 \|w\|_{L^{\frac{3}{\gamma},\infty}}\big)^{-1}\,,
	\]
	estimate (i) reads $ \|\Phi(u)\|_{L^\infty \widetilde{H}_\alpha^s}\leqslant M$ and shows that $\Phi$ maps the space $\mathcal{X}_{T,M}$ defined in \eqref{eq:fixed_point_space} into itself, whereas estimate (ii) reads $d(\Phi(u),\Phi(v))\leqslant\frac{1}{2}d(u,v)$ and shows that $\Phi$ is a contraction on $\mathcal{X}_{T,M}$. By Banach's fixed point theorem, there exists a unique fixed point $u\in \mathcal{X}_{T,M}$ of $\Phi$ and hence a unique solution $u\in \mathcal{X}_{T,M}$ to \eqref{eq:integral_formula_Duhamel}, which is therefore also continuous in time.

	Furthermore, by a customary continuation argument we can extend such a solution over a maximal interval for which the blow-up alternative holds  true. Also the continuous dependence on the initial data is a direct consequence of the fixed point argument. We omit the standard details, they are part of the well-established theory of semi-linear Schr\"odinger equations.
\end{proof}

We move now to the proof of Theorem \ref{thm:L2_reg}. Crucial for this case are the Strichartz estimates of Theorem \ref{thm:Strichartz}. 
To this aim, we modify the contraction space \eqref{eq:fixed_point_space} to the complete metric space $(\mathcal{Y}_{T,M},d)$ defined by
\begin{equation}\label{eq:fixed_point_space_Y}
\begin{split}
\mathcal{Y}_{T,M}\;&:=\;\left\{\!\! 
\begin{array}{c}
u\in L^\infty([-T,T],L^2(\mathbb{R}^3))\,\cap  L^{q(\gamma)}([-T,T],L^{r(\gamma)}(\mathbb{R}^3)) \\
\textrm{s.~t.}\;\|u\|_{L^\infty([-T,T],L^2(\mathbb{R}^3))}+\|u\|_{L^{q(\gamma)}([-T,T],L^{r(\gamma)}(\mathbb{R}^3))}\leqslant M
\end{array}
\!\!\right\} \\
d(u,v)\;&:=\;\|u-v\|_{L^\infty([-T,T],L^2(\mathbb{R}^3))}+\|u-v\|_{L^{q(\gamma)}([-T,T],L^{r(\gamma)}(\mathbb{R}^3))}
\end{split}
\end{equation}
for given $T,M>0$, where
\begin{equation}
q(\gamma)\;:=\;\frac{6}{\gamma}\,,\qquad r(\gamma)\;:=\;\frac{18}{9-2\gamma}
\end{equation}
are defined so as to form an admissible pair $(q(\gamma),r(\gamma))$ for $-\Delta_\alpha$, in the sense of \eqref{eq:admie}. For the rest of the proof let us drop the explicit dependence on $\gamma$ in $(q,r)$.

\begin{proof}[Proof of Theorem  \ref{thm:L2_reg}]~
	Clearly, when $\gamma=0$ the very same argument used in the proof of Theorem \ref{thm:low_reg} applies.

	When $\gamma\in(0,\frac{3}{2})$ we exploit instead a contraction argument in the modified space \eqref{eq:fixed_point_space_Y}.
	
	One has
	\[
	\begin{split}
	\|\Phi(u)&\|_{L^\infty L^2}+\|\Phi(u)\|_{L^q L^r}\;\leqslant\;\Big\|\int_0^t e^{\ii(t-\tau)\Delta_\alpha}\big((w*|u|^2)u\big)(\tau)\,\ud \tau \Big\|_{L^\infty L^2}\\
	&\qquad+\Big\|\int_0^t e^{\ii(t-\tau)\Delta_\alpha}\big((w*|u|^2)u\big)(\tau)\,\ud \tau \Big\|_{L^q L^r}+\|f\|_{L^2}+\|e^{\ii t\Delta_\alpha} f\|_{L^q L^r}\,,
	\end{split}
	\]
	from which, by means of the Strichartz estimates \eqref{stri-1}-\eqref{stri-2}, one deduces
	\[
	\|\Phi(u)\|_{L^\infty L^2}+\|\Phi(u)\|_{L^q L^r}\;\leqslant\;C\,\Big( \|f\|_{L^2}+\|(w*|u|^2)u\|_{L^{q'}L^{r'}}\Big)
	\]
	for some constant $C>0$.

	By H\"{o}lder's and Young's inequalities,
	\[
	\begin{split}
	\|(w*|u|^2)u\|_{L^{q'}L^{r'}}\;&\leqslant\; \|w*|u|^2\|_{L^{q'}L^{\frac{9}{\gamma}}}\|u\|_{L^\infty L^2}\\
	&\lesssim\;\|w\|_{L^{\frac{3}{\gamma},\infty}}\|u\|^2_{L^{2q'}L^{r}} \|u\|_{L^\infty L^2}
	\end{split}
	\]
	and
	\[
	\|u\|^2_{L^{2q'}L^{r}}\;\leqslant\;(2T)^{1-\frac{\gamma}{2}}\|u\|^2_{L^{q}L^{r}}\,,
	\]
	whence
	\[\tag{i}
	\|\Phi(u)\|_{L^\infty L^2}+\|\Phi(u)\|_{L^q L^r}\;\leqslant\;C_1\,\Big(\|f\|_{L^2}+T^{1-\frac{\gamma}{2}}\|w\|_{L^{\frac{3}{\gamma},\infty}}\|u\|^2_{L^{q}L^{r}}\|u\|_{L^\infty L^2}\Big)
	\]
	for some constant $C_1>0$.
	
	Following the very same scheme, one finds
	\[
	\begin{split}
	\|\Phi(u)&-\Phi(v)\|_{L^\infty L^2}+\|\Phi(u)-\Phi(v)\|_{L^q L^r}\;\lesssim\;\|(w*|u|^2)u-(w*|v|^2)v\|_{L^{q'}L^{r'}} \\
	&\leqslant\;\|(w*|u|^2)(u-v)\|_{L^{q'}L^{r'}}+\|w*(|u|^2-|v|^2)v\|_{L^{q'}L^{r'}} \,,
	\end{split}
	\]
	and moreover
	\[
	\|(w*|u|^2)(u-v)\|_{L^{q'}L^{r'}}\;\lesssim\;T^{1-\frac{\gamma}{2}}\|w\|_{L^{\frac{3}{\gamma},\infty}}\|u\|^2_{L^{q}L^{r}}\|u-v\|_{L^\infty L^2}
	\]
	and
	\[
	\|w*(|u|^2-|v|^2)v\|_{L^{q'}L^{r'}}\;\lesssim\;T^{1-\frac{\gamma}{2}}\|w\|_{L^{\frac{3}{\gamma},\infty}}\|u-v\|_{L^q L^r}\big(\|u\|_{L^q L^r}+\|v\|_{L^q L^r}\big)\|v\|_{L^\infty L^2}\,.
	\]
	Thus,
	\[\tag{ii}
	\begin{split}
	d(\Phi(u),\Phi(v))\;\leqslant\;&C_2\|w\|_{L^{\frac{3}{\gamma},\infty}}\big(\|u\|^2_{L^\infty L^2}+\|u\|^2_{L^q L^r}+\|v\|^2_{L^\infty L^2}+\|v\|^2_{L^q L^r}\big)\;\times \\
	&\times\;T^{1-\frac{\gamma}{2}}\,d(u,v)\,.
	\end{split}
	\]

	Therefore, choosing $T$ and $M$ such that 
	\[
	M\;=\;2C_1\,\|f\|_{L^2}\,,\qquad T\;=\;\big(8\max\{C_1,C_2\}M^2\|w\|_{L^{\frac{3}{\gamma},\infty}}\big)^{-1+\frac{\gamma}{2}}\,,
	\]
	estimate (i) reads $\|\Phi(u)\|_{L^\infty L^2}+\|\Phi(u)\|_{L^q L^r}\leqslant M$ and shows that $\Phi$ maps 
	$\mathcal{Y}_{T,M}$ into itself, whereas estimate (ii) reads $d(\Phi(u),\Phi(v))\leqslant\frac{1}{2}d(u,v)$ and shows that $\Phi$ is a contraction on $\mathcal{Y}_{T,M}$.

	The thesis then follows by Banach's fixed point theorem through the same arguments outlined in the end of the proof of Theorem \ref{thm:low_reg}.
\end{proof}

For later purposes, let us conclude this Section with the following stability result.

\begin{proposition}\label{stabint_mass}
	Let $\alpha\geqslant 0$. For given $w\in L^{\frac{3}{\gamma},\infty}(\mathbb{R}^3)$, $\gamma\in[0,\frac32)$, and $f\in L^2(\mathbb{R}^3)$, let $u$ be the unique strong $L^2$-solution to the Cauchy problem \eqref{eq:Cauchy_problem_sing_Hartree} in the maximal interval $(-T_*,T^*)$. Consider moreover the sequences $(w_n)_n$ and $(f_n)_n$ of potentials and initial data such that $w_n\xrightarrow[]{n\to+\infty} w$ in $L^{\frac{3}{\gamma},\infty}(\mathbb{R}^3)$ and $f_n\xrightarrow[]{n\to+\infty} f$ in $L^2(\mathbb{R}^3)$. Then there exists a time $T:=T(\|w\|_{L^{\frac{3}{\gamma},\infty}},\|f\|_{L^2})>0$, with $[-T,T]\subset(-T_*,T^*)$, such that, for sufficiently large $n$, the Cauchy problem \eqref{eq:Cauchy_problem_sing_Hartree} with potential $w_n$ and initial data $f_n$ admits a unique strong $L^2$-solution $u_n$ in the interval $[-T,T]$. Moreover,
	\begin{equation}\label{stabconvpa_mass}
	u_n\xrightarrow[]{n\to+\infty} u\qquad\textrm{in  }\;\;\mathcal{C}([-T,T],L^2(\R^3))\,.
	\end{equation}
\end{proposition}

\begin{proof}
	As a consequence of  Theorem \ref{thm:L2_reg}, there exist an interval $[-T_n,T_n]$ for some $T_n:=T_n(\|w_n\|_{L^{\frac{3}{\gamma},\infty}},\|f_n\|_{L^2})>0$ and a unique $u_n\in\mathcal{C}([-T_n,T_n],L^2(\mathbb{R}^3))$ such that
	\begin{equation*}\tag{i}
	u_n(t)\;=\;e^{\ii t\Delta_\alpha} f_n-\ii\int_0^t e^{\ii (t-\tau)\Delta_\alpha}(w_n*|u_n(\tau)|^2)u_n(\tau)\,\ud\tau\,.
	\end{equation*}
	Since $\|w_n\|_{L^{\frac{3}{\gamma},\infty}}$ and $\|f_n\|_{L^2}$ are asymptotically close, respectively, to $\|w\|_{L^{\frac{3}{\gamma},\infty}}$ and $\|f\|_{L^2}$, then there exists $T:=T(\|w\|_{L^{\frac{3}{\gamma},\infty}},\|f\|_{L^2})$ such that $T\leqslant T_n$ eventually in $n$, which means that $u_n$ is defined on  $[-T,T]$.
	Let us set $\phi_n:=u-u_n$.

	By assumption $u$ solves \eqref{eq:integral_formula_Duhamel}, thus subtracting (i) from \eqref{eq:integral_formula_Duhamel} yields
	\begin{equation*}\tag{ii}
	\begin{aligned}
	\phi_n\;&=\;e^{\ii t\Delta_\alpha} (f-f_n)-\ii\int_0^t e^{\ii (t-\tau)\Delta_\alpha}\big((w*|u|^2)u-(w_n*|u_n|^2)u_n\big)(\tau)\,\ud\tau\\
	&=\;e^{\ii t\Delta_\alpha} (f-f_n)-\ii\int_0^t e^{\ii (t-\tau)\Delta_\alpha}\Big\{\big((w-w_n)*|u|^2\big)u+(w_n*|u|^2)\phi_n\\
	&\qquad\qquad\qquad\qquad\qquad\qquad\qquad +\big(w_n*(\overline{u}\phi_n+\overline{\phi_n}u_n) \big)u_n\Big\}(\tau)\,\ud\tau\,.
	\end{aligned}
	\end{equation*}

	Let us first discuss the case $\gamma=0$. From (ii)  above, using H\"older's and Young's inequality in weak spaces, one has
	\[
	\begin{split}
	\|\phi_n\|_{L^{\infty}L^2}\;&\lesssim\;\|f-f_n\|_{L^2} +T\,\|w-w_n\|_{L^{\infty}}\,\|u\|^3_{L^{\infty}L^2} \\
	&\qquad +T\,\|w_n\|_{L^{\infty}}\big(\,\|u\|^2_{L^{\infty}L^2}+\|u_n\|^2_{L^{\infty}L^2}\,\big)\,\|\phi_n\|_{L^{\infty}L^2}\,.
	\end{split}
	\]
	Since $\|w_n\|_{L^{\infty}}$ and $\|u_n\|_{L^{\infty}L^2}$ are bounded uniformly in $n$, then the above inequality implies, decreasing further $T$ if needed,
	\[
	\|\phi_n\|_{L^{\infty}L^2}\;\lesssim\;\|f-f_n\|_{L^2}+\|w-w_n\|_{L^{\frac{3}{\gamma},\infty}}\;\xrightarrow[]{n\to+\infty}\;0\,,
	\]
	which proves the proposition in the case $\gamma=0$.

	Let now $\gamma\in(0,\frac32)$. In this case, owing to Theorem \ref{thm:L2_reg}, $u,u_n\in L^q([-T,T],L^r(\R^3))$, where $(q,r)=(\frac{6}{\gamma},\frac{18}{9-2\gamma})$ is the admissible pair defined in the proof therein. We can then argue as in the proof of Theorem \ref{thm:L2_reg}.
	Applying the Strichartz estimates \eqref{stri-1}-\eqref{stri-2} to the identity (ii) above, one gets
	\begin{equation*}\tag{iii}
	\begin{split}
	\|\phi_n\|_{L^{\infty}L^2}+\|&\phi_n\|_{L^{q}L^r}\;\lesssim\;\|f-f_n\|_{L^2}+\big\|\big((w-w_n)*|u|^2\big)u\big\|_{L^{q'}L^{r'}}\\
	&+\big\|(w_n*|u|^2)\phi_n\big\|_{L^{q'}L^{r'}}+\big\|\big(|w_n|*(|u_n|+|u|)|\phi_n|\big)u_n\big\|_{L^{q'}L^{r'}}\,.
	\end{split}
	\end{equation*}
	By means of H\"older's and Young's inequality in weak spaces one finds
	\[\tag{iv}
	\begin{split}
	\big\|\big((w-w_n)*|u|^2\big)u\big\|_{L^{q'}L^{r'}}\;&\lesssim\;T^{1-\frac{\gamma}{2}}\|w-w_n\|_{L^{\frac{3}{\gamma},\infty}}\|u_n\|^2_{L^qL^r}\,\|u\|_{L^{\infty}L^2} \\
	\big\|(w_n*|u|^2)\phi_n\big\|_{L^{q'}L^{r'}}\;&\lesssim\;T^{1-\frac{\gamma}{2}}\|w_n\|_{L^{\frac{3}{\gamma},\infty}}\,\|u\|^2_{L^qL^r}\|\phi_n\|_{L^{\infty}L^2} \\
	\big\|\big(|w_n|*(|u_n|+|u|)|\phi_n|\big)u_n\big\|_{L^{q'}L^{r'}}\;&\lesssim\;T^{1-\frac{\gamma}{2}}\|w_n\|_{L^{\frac{3}{\gamma},\infty}}\,\big(\,\|u_n\|_{L^qL^r}+\|u\|_{L^qL^r}\big)\;\times \\
	&\qquad\qquad\times\|\phi_n\|_{L^qL^r}\|u\|_{L^{\infty}L^2}\,.
	\end{split}
	\]
	Since $\|w_n\|_{L^{\frac{3}{\gamma},\infty}}$ and $\|u_n\|_{L^qL^r}$ are bounded uniformly in $n$, then inequalities (iii) and (iv) imply, decreasing further $T$ if needed,
	\[
	\|\phi_n\|_{L^{\infty}L^2}+\|\phi_n\|_{L^{q}L^r}\;\lesssim\;\|f-f_n\|_{L^2}+\|w-w_n\|_{L^{\frac{3}{\gamma},\infty}}\;\xrightarrow[]{n\to+\infty}\;0\,,
	\]
	which completes the proof.
\end{proof}

\section{Intermediate regularity theory}\label{sec:intermediate}

In this Section we prove Theorem \ref{thm:high_reg}. The proof is based again on a contraction argument in the complete metric space $\mathcal{X}_{T,M}$, for suitable $T,M>0$, defined in \eqref{eq:fixed_point_space}, now with $s\in(\frac{1}{2},\frac{3}{2})$.

As a consequence, in the energy space ($s=1$) we shall deduce that the solution to the integral problem \eqref{eq:integral_formula_Duhamel} is also a solution to the differential problem \eqref{eq:Cauchy_problem_sing_Hartree}.

We conclude the Section with a stability result of the solution with respect to the initial datum $f$ and the potential $w$.

Let us start with two preparatory lemmas.

\begin{lemma}\label{convu}
	Let $\alpha\geqslant 0$ and $s\in(\frac{1}{2},\frac{3}{2})$. Let $w\in W^{s,p}(\mathbb{R}^3)$ for $p\in(2,+\infty)$. Then
	\begin{equation}\label{eq:convu}
	\begin{split}
	\|w*(\psi_1\psi_2)\|_{L^\infty(\R^3)}\;&\lesssim\;\|w*(\psi_1\psi_2)\|_{W^{s,3p}(\R^3)} \\
	&\lesssim\;\|w\|_{W^{s,p}(\R^3)}\|\psi_1\|_{ \widetilde{H}^s_\alpha(\mathbb{R}^3)}\|\psi_2\|_{ \widetilde{H}^s_\alpha(\mathbb{R}^3)}
	\end{split}
	\end{equation}
	for any $ \widetilde{H}^s_\alpha$-functions $\psi_1$, $\psi_2$, and $\psi_3$.
\end{lemma}

\begin{proof}
	The first inequality in \eqref{eq:convu} is due to Sobolev's embedding
	\[
	W^{s,3p}(\mathbb{R}^3)\;\hookrightarrow\;L^\infty(\mathbb{R}^3)\,.
	\]
	For the second inequality, let us observe preliminarily that
	\begin{equation*}
	\widetilde{H}^s_\alpha(\mathbb{R}^3)\;\hookrightarrow\;L^{\frac{6p}{3p-2}}(\R^3)\,. 
	\end{equation*}

	Indeed, decomposing by means of \eqref{eq:Ds_s1/2-3/2} a generic $\psi\in\widetilde{H}^s_\alpha(\mathbb{R}^3)$ as $\psi=\phi_\lambda+\kappa_\lambda\,G_\lambda$ for some  $\phi_\lambda\in H^s(\R^3)$ and some $\kappa_\lambda\in\C$, one has 
	\[
	\begin{split}
	\|\psi\|_{L^{\frac{6p}{3p-2}}}\;&\leqslant\; \|\phi_\lambda\|_{L^{\frac{6p}{3p-2}}}+|\kappa_\lambda|\,\|G_\lambda\|_{L^{\frac{6p}{3p-2}}} \;\lesssim\; \|\phi_\lambda\|_{H^s}+|\kappa_\lambda|\;\approx\; \|\psi\|_{\widetilde{H}^s_\alpha}\,,
	\end{split}
	\]
	the second step following from Sobolev's embedding $H^s(\R^3)\hookrightarrow L^{\frac{6p}{3p-2}}(\R^3)$  and from $G_\lambda\in L^{\frac{6p}{3p-2}}(\R^3)$, because  $\frac{6p}{3p-2}\in[2,3)$ for $p\in(2,+\infty)$, the last step being the norm equivalence \eqref{eq:equiv_of_norms_s1232}. Therefore Young's inequality yields
	\begin{equation*}
	\begin{aligned}
	\|w*(\psi_1\psi_2)\|_{W^{s,3p}}\;&\approx\;\|(\mathbbm{1}-\Delta)^{\frac s2}(w*(\psi_1\psi_2))\|_{L^{3p}}\;=\;\|((\mathbbm{1}-\Delta)^{\frac s2}w)*(\psi_1\psi_2)\|_{L^{3p}}\\
	&\lesssim \;\|(\mathbbm{1}-\Delta)^{\frac s2}w\|_{L^p}\:\|\psi_1\|_{L^{\frac{6p}{3p-2}}}\,\|\psi_2\|_{L^{\frac{6p}{3p-2}}} \\
	&\lesssim\;\|w\|_{W^{s,p}(\R^3)}\|\psi_1\|_{ \widetilde{H}^s_\alpha(\mathbb{R}^3)}\|\psi_2\|_{ \widetilde{H}^s_\alpha(\mathbb{R}^3)}\,,
	\end{aligned}
	\end{equation*}
	thus proving \eqref{eq:convu}.
\end{proof}

\begin{lemma}\label{produ}
	Let $\alpha\geqslant 0$ and $s\in(\frac{1}{2},\frac{3}{2})$. Let $h\in W^{s,q}(\mathbb{R}^3)$ for $q\in(6,+\infty)$. Then $h\psi\in \widetilde{H}^s_\alpha(\mathbb{R}^3)$ for each $\psi\in \widetilde{H}^s_\alpha(\mathbb{R}^3)$ and
	\begin{equation}
	\|h\,\psi\|_{\widetilde{H}^s_\alpha(\mathbb{R}^3)}\;\lesssim\;\|h\|_{W^{s,q}(\R^3)}\:\|\psi\|_{\widetilde{H}^s_\alpha(\mathbb{R}^3)}\,.
	\end{equation}
\end{lemma}

\begin{proof}
	Let us decompose  $\psi\in  \widetilde{H}^s_\alpha(\mathbb{R}^3)$ as  $\psi=\phi_\lambda+\kappa_\lambda G_\lambda$ for some $\phi_\lambda\in H^s(\R^3)$ and $\kappa_\lambda\in\C$, according to \eqref{eq:Ds_s1/2-3/2}. 
	On the other hand, by the embedding \eqref{eq:SobEmb-I} the function $h$ is continuous and
	$|h(0)|\leqslant\|h\|_{L^\infty(\mathbb{R}^3)}\lesssim\|h\|_{W^{s,q}(\R^3)}$. Thus,
	\begin{equation*}\tag{i}
	h\,g\;=\;h\,\phi_\lambda+\kappa_\lambda\,(h-h(0))\,G_\lambda+\kappa_\lambda\,h(0)\,G_\lambda\,.
	\end{equation*}
	Applying the fractional Leibniz rule \eqref{eq:KatoPonce} and using Sobolev's embedding,
	\begin{equation*}\tag{ii}
	\begin{aligned}
	\|h\,\phi_\lambda\|_{H^s}\;&\approx\;\|(\mathbbm{1}-\Delta)^{\frac s2}(h\,\phi_\lambda)\|_{L^2} \\
	&\lesssim\; \|(\mathbbm{1}-\Delta)^{\frac s2}h\|_{L^q}\,\|\phi_\lambda\|_{L^{\frac{2q}{q-2}}}+\|h\|_{L^{\infty}}\|(\mathbbm{1}-\Delta)^{\frac s2}\phi_\lambda\|_{L^2} \\
	&\lesssim\;\|h\|_{W^{s,q}}\|\phi_\lambda\|_{H^s}\,.
	\end{aligned}
	\end{equation*}
	Moreover, since $G_\lambda\in L^2(\R^3)$,
	$$\|(h-h(0))G_\lambda\|_{L^2}\;\lesssim\;\|h-h(0)\|_{L^{\infty}}\;\lesssim\;\|h\|_{W^{s,q}}\,;$$
	this, together with the estimate \eqref{tecre}, gives
	\begin{equation*}\tag{iii}
	\|\kappa_\lambda\,(h-h(0))G_\lambda\|_{H^s}\;\lesssim\; |\kappa_\lambda|\,\|h\|_{W^{s,q}}\,.
	\end{equation*}
	The bounds (ii) and (iii) imply that $h\psi$ is the sum of the function $h\phi_\lambda+\kappa_\lambda(h-h(0))G_\lambda\in H^s(\mathbb{R}^3)$ and of the multiple $\kappa_\lambda h(0)G_\lambda$ of $G_\lambda$: as such, owing to \eqref{eq:Ds_s1/2-3/2}, $h\psi$ belongs to $\widetilde{H}^s_\alpha(\mathbb{R}^3)$ and its $\widetilde{H}^s_\alpha$-norm is estimated, according to the norm equivalence \eqref{eq:equiv_of_norms_s1232}, by
	\[
	\begin{split}
	\|h\,\psi\|_{\widetilde{H}^s_\alpha(\mathbb{R}^3)}\;&\approx\;\|h\,\phi_\lambda+\kappa_\lambda\,(h-h(0))\,G_\lambda\|_{\widetilde{H}^s_\alpha}+|\kappa_\lambda|\,|h(0)| \\
	&\lesssim\;\|h\|_{W^{s,q}}(\|\phi_\lambda\|_{H^s}+|\kappa_\lambda|)+ |\kappa_\lambda|\,\|h\|_{W^{s,q}} \\
	&\approx\;\|h\|_{W^{s,q}}\,\|\psi\|_{\widetilde{H}^s_\alpha}\,,
	\end{split}
	\]
	which completes the proof.
\end{proof}

Combining Lemmas \ref{convu} and \ref{produ} one therefore has the trilinear estimate
\begin{equation}\label{eq:LaL}
\|(w*(u_1u_2))u_3\|_{\widetilde{H}_\alpha^s(\mathbb{R}^3) }\;\lesssim\;\|w\|_{W^{s,p}(\R^3)}\prod_{j=1}^3\|u_j\|_{\widetilde{H}_\alpha^s(\mathbb{R}^3) }\,.
\end{equation}

Let us now prove Theorem \ref{thm:high_reg}.

\begin{proof}[Proof of Theorem \ref{thm:high_reg}]~
	
	From the expression \eqref{eq:Smap} for the solution map $\Phi(u)$ and from the bound \eqref{eq:LaL}  one finds
	\[\tag{i}
	\begin{split}
	\|\Phi(u)\|_{L^\infty\widetilde{H}_\alpha^s }\;&\leqslant\;\|f\|_{\widetilde{H}_\alpha^s }+T\,\|(w*|u|^2)u\|_{L^\infty\widetilde{H}_\alpha^s } \\
	&\leqslant\;\|f\|_{\widetilde{H}_\alpha^s }+C_1\,T\,\|w\|_{W^{s,p}}\,\|u\|^3_{L^{\infty}\widetilde{H}^s_\alpha}
	\end{split}
	\]
	for some constant $C_1>0$.
	
	Moreover,
	\begin{equation*}\tag{ii}
	\begin{split}
	\|\Phi(u&)-\Phi(v)\|_{L^\infty L^2 }\;\leqslant\;T\,\|(w*|u|^2)u-(w*|v|^2)v\|_{L^\infty L^2 } \\
	&\lesssim\;T\,\big( \|(w*|u|^2)(u-v)\|_{L^\infty L^2 } + \big\|\big(w*(|u|^2-|v|^2\big) v\big\|_{L^\infty L^2 }\big) \,.
	\end{split} 
	\end{equation*}
	For the first summand in the r.h.s.~above estimate \eqref{eq:convu} and H\"older's inequality yield
	\begin{equation*}\tag{iii}
	\begin{split}
	\|(w*|u|^2)(u-v)\|_{L^\infty L^2 }\;&\leqslant\; \|w*|u|^2\|_{L^{\infty}L^{\infty}}\,\|u-v\|_{L^{\infty}L^2}\\
	&\lesssim\; \|w\|_{W^{s,p}}\|u\|^2_{L^{\infty}\widetilde{H}_{\alpha}^s}\,\|u-v\|_{L^{\infty}L^2}\,.
	\end{split}
	\end{equation*}
	For the second summand, let us observe preliminarily that
	\begin{equation*}\tag{iv}
	\widetilde{H}^s_\alpha(\mathbb{R}^3)\;\hookrightarrow\;L^{3,\infty}(\R^3)\,. 
	\end{equation*}
	Indeed, decomposing by means of \eqref{eq:Ds_s1/2-3/2} a generic $\psi\in\widetilde{H}^s_\alpha(\mathbb{R}^3)$ as $\psi=\phi_\lambda+\kappa_\lambda\,G_\lambda$ for some  $\phi_\lambda\in H^s(\R^3)$ and some $\kappa_\lambda\in\C$, one has 
	\[
	\begin{split}
	\|\psi\|_{L^{3,\infty}}\;&\leqslant\; \|\phi_\lambda\|_{L^{3,\infty}}+|\kappa_\lambda|\,\|G_\lambda\|_{L^{3,\infty}} \lesssim \|\phi_\lambda\|_{H^s}+|\kappa_\lambda|\;\approx\; \|\psi\|_{\widetilde{H}^s_\alpha}\,,
	\end{split}
	\]
	the second step following from the Sobolev's embedding $H^s(\R^3)\hookrightarrow L^{3}(\R^3)$, the last step being the norm equivalence \eqref{eq:equiv_of_norms_s1232}.
	Then (iv) above, Sobolev's embedding $W^{s,p}(\R^3)\hookrightarrow L^3(\R^3)$, and an application of Holder's and Young's inequality in Lorentz spaces, yield
	\begin{equation*}\tag{v}
	\begin{split}
	\|(w*(|u|^2\!-\!|v|^2))&v\|_{L^\infty L^2 }\;\leqslant\;\|w*(|u|^2\!-\!|v|^2)\|_{L^{\infty}L^{6,2}}\,\|v\|_{L^{\infty}L^{3,\infty}}\\
	&\lesssim\; \|w\|_{L^3}\,\|u+v\|_{L^{\infty}L^{3,\infty}}\,\|u-v\|_{L^{\infty}L^2}\,\|v\|_{L^{\infty}L^{3,\infty}}\\
	&\lesssim\; \|w\|_{W^{s,p}}\,\|u+v\|_{L^{\infty}\widetilde{H}_{\alpha}^s}\,\|u-v\|_{L^{\infty}L^2}\,\|v\|_{L^{\infty}\widetilde{H}_{\alpha}^s}\,.
	\end{split}
	\end{equation*}
	Thus, (ii), (iii), and (v) together give
	\[\tag{vi}
	\begin{split}
	d(\Phi(u),\Phi(v))\;&\leqslant\;C_2\,T\,\|w\|_{W^{s,p}}\,\big(\,\|u\|^2_{L^\infty \widetilde{H}_\alpha^s }+\|v\|^2_{L^\infty \widetilde{H}_\alpha^s} \big)\,d(u,v)
	\end{split}
	\]
	for some constant $C_2>0$. 
	
	Now, setting $C:=\max\{C_1,C_2\}$ and choosing $T$ and $M$ such that
	\[
	M\;=\;2\,\|f\|_{\widetilde{H}_\alpha^s}\,,\qquad T\;=\;{\textstyle\frac{1}{4}}\,\big(C\,M^2 \|w\|_{W^{s,p}}\big)^{-1}\,,
	\]
	estimate (i) reads $ \|\Phi(u)\|_{L^\infty \widetilde{H}_\alpha^s}\leqslant M$ and shows that $\Phi$ maps the space $\mathcal{X}_{T,M}$ defined in \eqref{eq:fixed_point_space} into itself, whereas estimate (vi) reads $d(\Phi(u),\Phi(v))\leqslant\frac{1}{2}d(u,v)$ and shows that $\Phi$ is a contraction on $\mathcal{X}_{T,M}$. By Banach's fixed point theorem, there exists a unique fixed point $u\in \mathcal{X}_{T,M}$ of $\Phi$ and hence a unique solution $u\in \mathcal{X}_{T,M}$ to \eqref{eq:integral_formula_Duhamel}, which is therefore also continuous in time.

	Furthermore, by a standard continuation argument we can extend such a solution over a maximal interval for which the blow-up alternative holds  true. Also the continuous dependence on the initial data is a direct consequence of the fixed point argument.
\end{proof}

A straightforward consequence of Theorem \ref{thm:high_reg} when $s=1$ concerns the differential meaning of the local strong solution determined so far.

\begin{corollary}[Integral and differential formulation]\label{soddire}
	Let $\alpha\geqslant 0$. For given $w\in W^{1,p}(\mathbb{R}^3)$, $p\in(2,+\infty)$, and $f\in \widetilde{H}^1_\alpha(\mathbb{R}^3)$, let  $u$ be the unique solution in the class $\mathcal{C}([-T,T],\widetilde{H}^{1}_\alpha(\mathbb{R}^3))$ to the integral equation \eqref{eq:integral_formula_Duhamel} in the interval $[-T,T]$ for some $T>0$, as given by Theorem \ref{thm:high_reg}.
	Then $u(0)=f$ and $u$ satisfies the differential equation \eqref{eq:sing_Hartree}
	as an identity between $\widetilde{H}^{-1}_\alpha$-functions, $\widetilde{H}^{-1}_\alpha(\mathbb{R}^3)$ being the topological dual of $\widetilde{H}^{1}_\alpha(\mathbb{R}^3)$. 
\end{corollary}

\begin{proof}
	The bound \eqref{eq:LaL} shows that the non-linearity defines a map $u\mapsto (w*|u|^2)u$ that is continuous from $\widetilde{H}^{1}_\alpha(\mathbb{R}^3)$ into itself, and hence in particular it is continuous from $\widetilde{H}^{1}_\alpha(\mathbb{R}^3)$ to $\widetilde{H}^{-1}_\alpha(\mathbb{R}^3)$. Then the thesis follows by standard facts of the theory of linear semi-groups (see \cite[Section 1.6]{cazenave}).
\end{proof}

For later purposes, let us conclude this Section with the following stability result.

\begin{proposition}\label{stabint}
	Let $\alpha\geqslant 0$ and $s\in(\frac{1}{2},\frac{3}{2})$. For given $w\in W^{s,p}(\mathbb{R}^3)$, $p\in(2,+\infty)$, and $f\in \widetilde{H}^{s}_\alpha(\mathbb{R}^3)$, let $u$ be the unique strong $\widetilde{H}^{s}_\alpha$-solution to the Cauchy problem \eqref{eq:Cauchy_problem_sing_Hartree} in the maximal interval $(-T_*,T^*)$. Consider moreover the sequences $(w_n)_n$ and $(f_n)_n$ of potentials and initial data such that $w_n\xrightarrow[]{n\to+\infty} w$ in $W^{s,p}(\mathbb{R}^3)$ and $f_n\xrightarrow[]{n\to+\infty} f$ in $\widetilde{H}^{s}_\alpha(\mathbb{R}^3)$. Then there exists a time $T:=T(\|w\|_{W^{s,p}},\|f\|_{\widetilde{H}^{s}_\alpha})>0$, with $[-T,T]\subset(T_*,T^*)$, such that, for sufficiently large $n$, the Cauchy problem \eqref{eq:Cauchy_problem_sing_Hartree} with potential $w_n$ and initial data $f_n$ admits a unique strong $\widetilde{H}^{s}_\alpha$-solution $u_n$ in the interval $[-T,T]$. Moreover,
	\begin{equation}\label{stabconvpa}
	u_n\xrightarrow[]{n\to+\infty} u\qquad\textrm{in  }\;\;\mathcal{C}([-T,T],\widetilde{H}^{s}_\alpha(\R^3))\,.
	\end{equation}
\end{proposition}

\begin{proof}
	As a consequence of  Theorem \ref{thm:high_reg}, there exist an interval $[-T_n,T_n]$ for some $T_n:=T_n(\|w_n\|_{W^{s,p}},\|f_n\|_{\widetilde{H}^{s}_\alpha})>0$ and a unique $u_n\in\mathcal{C}([-T_n,T_n],\widetilde{H}^{1}_\alpha(\mathbb{R}^3))$ such that
	\begin{equation*}\tag{*}
	u_n(t)\;=\;e^{\ii t\Delta_\alpha} f_n-\ii\int_0^t e^{\ii (t-\tau)\Delta_\alpha}(w_n*|u_n(\tau)|^2)u_n(\tau)\,\ud\tau\,.
	\end{equation*}
	Since $\|w_n\|_{W^{s,p}}$ and $\|f_n\|_{\widetilde{H}^{s}_\alpha}$ are asymptotically close, respectively, to $\|w\|_{W^{s,p}}$ and $\|f\|_{\widetilde{H}^{s}_\alpha}$, then there exists $T:=T(\|w\|_{W^{s,p}},\|f\|_{\widetilde{H}^{s}_\alpha})$ such that $T\leqslant T_n$ eventually in $n$, which means that $u_n$ is defined on  $[-T,T]$. Let us set $\phi_n:=u-u_n$. By assumption $u$ solves \eqref{eq:integral_formula_Duhamel}, thus subtracting (*) from \eqref{eq:integral_formula_Duhamel} yields
	\begin{equation*}
	\begin{aligned}
	\phi_n\;&=\;e^{\ii t\Delta_\alpha} (f-f_n)-\ii\int_0^t e^{\ii (t-\tau)\Delta_\alpha}\big((w*|u|^2)u-(w_n*|u_n|^2)u_n\big)(\tau)\,\ud\tau\\
	&=\;e^{\ii t\Delta_\alpha} (f-f_n)-\ii\int_0^t e^{\ii (t-\tau)\Delta_\alpha}\Big\{\big((w-w_n)*|u|^2\big)u+(w_n*|u|^2)\phi_n\\
	&\qquad\qquad\qquad\qquad\qquad\qquad\qquad +\big(w_n*(\overline{u}\phi_n+\overline{\phi_n}u_n) \big)u_n\Big\}(\tau)\,\ud\tau\,.
	\end{aligned}
	\end{equation*}
	From the above identity, taking the $\widetilde{H}^{s}_\alpha$-norm of $\phi_n$ boils down to repeatedly applying the estimate \eqref{eq:LaL} to the summands in the integral on the r.h.s., thus yielding 
	\[
	\begin{split}
	\|\phi_n\|_{L^{\infty}\widetilde{H}^{s}_\alpha}\;&\lesssim\;\|f-f_n\|_{\widetilde{H}^{s}_\alpha} +T\,\|w-w_n\|_{W^{s,p}}\,\|u\|^3_{L^{\infty}\widetilde{H}^{s}_\alpha} \\
	&\qquad +T\,\|w_n\|_{W^{s,p}}\big(\,\|u\|^2_{L^{\infty}\widetilde{H}^{s}_\alpha}+\|u_n\|^2_{L^{\infty}\widetilde{H}^{s}_\alpha}\,\big)\,\|\phi_n\|_{L^{\infty}\widetilde{H}^{s}_\alpha}\,.
	\end{split}
	\]
	Since by assumption $\|w_n\|_{W^{s,p}}$ and $\|u_n\|^2_{L^{\infty}\widetilde{H}^{s}_\alpha}$ are bounded uniformly in $n$, then the above inequality implies, decreasing further $T$ if needed,
	\[
	\|\phi_n\|_{L^{\infty}\widetilde{H}^{s}_\alpha}\;\lesssim\;\|f-f_n\|_{\widetilde{H}^{s}_\alpha}+\|w-w_n\|_{W^{s,p}}\;\xrightarrow[]{n\to+\infty}\;0\,,
	\]
	which completes the proof.
\end{proof}

\section{High regularity Theory}\label{sec:higher_reg}

In this Section we prove Theorem \ref{highh} for the regime $s\in(\frac{3}{2},2]$. The approach is again a contraction argument, that we now set in the spherically symmetric sector of the space $\mathcal{X}_{T,M}$ introduced in \eqref{eq:fixed_point_space}, namely in the complete metric space $(\mathcal{X}^{(0)}_{T,M},d)$ with
\begin{equation}\label{eq:fixed_point_space_rad}
\begin{split}
\mathcal{X}^{(0)}_{T,M}\;&:=\;\big\{u\in L^\infty([-T,T],\widetilde{H}_{\alpha,\mathrm{rad}}^s(\mathbb{R}^3))\,|\,\|u\|_{L^\infty([-T,T],\widetilde{H}_\alpha^s(\mathbb{R}^3)}\leqslant M\big\} \\
d(u,v)\;&:=\;\|u-v\|_{L^\infty([-T,T],L^2(\mathbb{R}^3))}
\end{split}
\end{equation}
for suitable $T,M>0$.

A very much useful by-product of such a contraction argument will be the proof that when $s=2$ the solution to the integral problem \eqref{eq:integral_formula_Duhamel} is also a solution to the differential problem \eqref{eq:Cauchy_problem_sing_Hartree}, as we shall show in a moment.

Let us start with two preparatory lemmas.
\begin{lemma}\label{convu_bou}
	Let $\alpha\geqslant 0$ and $s\in(\frac{3}{2},2]$. Let $w\in W^{s,p}(\mathbb{R}^3)$ for $p\in(2,+\infty)$ and assume that $w$ is spherically symmetric. Then
	\begin{itemize}
		\item[(i)] one has the estimate
		\begin{equation}\label{eq:convu_bou}
		\begin{split}
		\|w*(\psi_1\psi_2)\|_{L^\infty(\R^3)}\;&\lesssim\;\|w*(\psi_1\psi_2)\|_{W^{s,3p}(\R^3)} \\
		&\lesssim\;\|w\|_{W^{s,p}(\R^3)}\|\psi_1\|_{ \widetilde{H}^s_\alpha(\mathbb{R}^3)}\|\psi_2\|_{ \widetilde{H}^s_\alpha(\mathbb{R}^3)}
		\end{split}
		\end{equation}
		for any $ \widetilde{H}^s_\alpha$-functions $\psi_1$, $\psi_2$, and $\psi_3$;
		\item[(ii)] if in addition $\psi_1$, $\psi_2$ are spherically symmetric, so too is $w*(\psi_1\psi_2)$ and
		$$\big(\nabla(w*(\psi_1\psi_2))\big)(0)\;=\;0\,.$$
	\end{itemize}
\end{lemma}

\begin{proof}
	(i) The first inequality in \eqref{eq:convu_bou} is due to Sobolev's embedding
	\[
	W^{s,3p}(\mathbb{R}^3)\;\hookrightarrow\;L^\infty(\mathbb{R}^3)\,.
	\]
	For the second inequality, let us observe preliminarily that
	\[
	\widetilde{H}^s_\alpha(\mathbb{R}^3)\;\hookrightarrow\; L^{\frac{6p}{3p-2}}(\R^3)\,.
	\]
	Indeed, decomposing by means of \eqref{eq:Ds_s3/2-2} a generic $\psi\in\widetilde{H}^s_\alpha(\mathbb{R}^3)$ as $\psi=\phi_{\lambda}+\frac{\phi_{\lambda}(0)}{\alpha+\frac{\sqrt{\lambda}}{4\pi}}G_{\lambda}$ for  some $\phi_{\lambda}\in H^s(\R^3)$, one has
	\[
	\begin{split}
	\|\psi\|_{L^{\frac{6p}{3p-2}}}\;&\lesssim\; \|\phi_\lambda\|_{L^{\frac{6p}{3p-2}}}+|\phi_{\lambda}(0)|\,\|G_\lambda\|_{L^{\frac{6p}{3p-2}}}\; \lesssim \;\|\phi_\lambda\|_{H^s}\;\approx\; \|\psi\|_{\widetilde{H}^s_\alpha}\,,
	\end{split}
	\]
	the second step following from Sobolev's embedding $H^s(\R^3)\hookrightarrow L^{\frac{6p}{3p-2}}(\R^3)\cap  L^{\infty}(\R^3) $  and from $G_\lambda\in L^{\frac{6p}{3p-2}}(\R^3)$, because  $\frac{6p}{3p-2}\in[2,3)$ for $p\in(2,+\infty)$, the last step being the norm equivalence \eqref{eq:equiv_of_norms_s322}. Therefore Young's inequality yields
	\begin{equation*}
	\begin{aligned}
	\|w*(\psi_1\psi_2)\|_{W^{s,3p}}\;&\approx\;\|(\mathbbm{1}-\Delta)^{\frac s2}(w*(\psi_1\psi_2))\|_{L^{3p}} \\
	&=\;\|((\mathbbm{1}-\Delta)^{\frac s2}w)*(\psi_1\psi_2)\|_{L^{3p}}\\
	&\lesssim \;\|(\mathbbm{1}-\Delta)^{\frac s2}w\|_{L^p}\:\|\psi_1\|_{L^{\frac{6p}{3p-2}}}\,\|\psi_2\|_{L^{\frac{6p}{3p-2}}} \\
	&\lesssim\;\|w\|_{W^{s,p}(\R^3)}\,\|\psi_1\|_{ \widetilde{H}^s_\alpha(\mathbb{R}^3)}\,\|\psi_2\|_{ \widetilde{H}^s_\alpha(\mathbb{R}^3)}\,,
	\end{aligned}
	\end{equation*}
	thus proving \eqref{eq:convu_bou}.

	(ii) The spherical symmetry of $w*(\psi_1\psi_2)$ in this second case is obvious.
	From Sobolev's embedding $W^{s,3p}(\R^3)\hookrightarrow \mathcal{C}^1(\R^3)$ we deduce that $\nabla(w*(\psi_1\psi_2))(x)$ is well defined for every $x\in\R^3$; moreover, 
	$$\nabla(w*(\psi_1\psi_2))(0)\;=\;\big((\nabla w)*(\psi_1\psi_2)\big)(0)\;=\;\int_{\R^3}(\nabla w)(-y)\psi_1(y)\psi_2(y)\,\ud y\;=\;0\,,$$
	the above integral vanishing because the integrand is of the form $R(y)\frac{y}{|y|}$ for some spherically symmetric function $R$.
\end{proof}

\begin{lemma}\label{produ_bou}
	Let $\alpha\geqslant 0$ and $s\in(\frac{3}{2},2]$. Let $h\in W_{\mathrm{rad}}^{s,q}(\mathbb{R}^3)$ for some $q\in(6,+\infty)$ and assume that $(\nabla h)(0)=0$. Then $h\psi\in \widetilde{H}^s_\alpha(\mathbb{R}^3)$ for each $\psi\in \widetilde{H}^s_\alpha(\mathbb{R}^3)$ and
	\begin{equation}
	\|h\psi\|_{\widetilde{H}^s_\alpha(\mathbb{R}^3)}\;\lesssim\;\|h\|_{W^{s,q}(\R^3)}\,\|\psi\|_{\widetilde{H}^s_\alpha(\mathbb{R}^3)}\,.
	\end{equation}
\end{lemma}

\begin{proof}
	Let us decompose  $\psi\in \widetilde{H}^s_\alpha(\mathbb{R}^3)$ as  $\psi=\phi_{\lambda}+\frac{\phi_{\lambda}(0)}{\alpha+\frac{\sqrt{\lambda}}{4\pi}}G_{\lambda}$  for some $\phi_{\lambda}\in H^s(\R^3)$, according to \eqref{eq:Ds_s3/2-2}. On the other hand, by the embedding \eqref{eq:SobEmb-II} the function $h$ is continuous and $|h(0)|\leqslant\|h\|_{L^\infty(\mathbb{R}^3)}\lesssim\|h\|_{W^{s,q}(\R^3)}$. Thus,
	\begin{equation*}\tag{i}
	h\,\psi\;=\;h\,\phi_\lambda+{\textstyle\frac{\phi_{\lambda}(0)}{\alpha+\frac{\sqrt{\lambda}}{4\pi}}}\,(h-h(0))\,G_\lambda+{\textstyle\frac{\phi_{\lambda}(0)}{\alpha+\frac{\sqrt{\lambda}}{4\pi}}}\,h(0)\,G_\lambda\,.
	\end{equation*}
	Applying the fractional Leibniz rule \eqref{eq:KatoPonce} and using Sobolev's embedding,
	\begin{equation*}\tag{ii}
	\begin{aligned}
	\|h\,\phi_\lambda\|_{H^s}\;&\approx\;\|(\mathbbm{1}-\Delta)^{\frac s2}(h\,\phi_\lambda)\|_{L^2} \\
	&\lesssim\; \|(\mathbbm{1}-\Delta)^{\frac s2}h\|_{L^q}\,\|\phi_\lambda\|_{L^{\frac{2q}{q-2}}}+\|h\|_{L^{\infty}}\|(\mathbbm{1}-\Delta)^{\frac s2}\phi_\lambda\|_{L^2} \\
	&\lesssim\;\|h\|_{W^{s,q}}\|\phi_\lambda\|_{H^s}\,.
	\end{aligned}
	\end{equation*}
	Moreover, since $G_\lambda\in L^2(\R^3)$,
	$$\big\|{\textstyle\frac{\phi_{\lambda}(0)}{\alpha+\frac{\sqrt{\lambda}}{4\pi}}}(h-h(0))G_\lambda\big\|_{L^2}\;\lesssim\;\|h-h(0)\|_{L^{\infty}}\,\|\phi_\lambda\|_{L^\infty}\;\lesssim\;\|h\|_{W^{s,q}}\,\|\phi_\lambda\|_{H^s}\,;$$
	this, together with the estimate \eqref{tecredue} (which requires indeed spherical symmetry), gives
	\begin{equation*}\tag{iii}
	\big\|{\textstyle\frac{\phi_{\lambda}(0)}{\alpha+\frac{\sqrt{\lambda}}{4\pi}}}\,(h-h(0))G_\lambda\big\|_{H^s}\;\lesssim\; \|h\|_{W^{s,q}}\|\phi_\lambda\|_{H^s}\,.
	\end{equation*}
	The bounds (ii) and (iii) above imply that $F_\lambda:=h\phi_\lambda+{\textstyle\frac{\phi_{\lambda}(0)}{\alpha+\frac{\sqrt{\lambda}}{4\pi}}}(h-h(0))G_\lambda$ belongs to $ H^s(\mathbb{R}^3)$ with
	\[\tag{iv}
	\|F\|_{H^s}\;\lesssim\;\|h\|_{W^{s,q}}\|\phi_\lambda\|_{H^s}\,.
	\]
	In particular, $F_\lambda$ is continuous. One has
	\[
	\begin{split}
	F_\lambda(0)\;=\;h(0)\phi_\lambda(0)+{\textstyle\frac{\phi_{\lambda}(0)}{\alpha+\frac{\sqrt{\lambda}}{4\pi}}}\lim_{|x|\to 0}\frac{h(x)-h(0)}{|x|}\;=\;h(0)\phi_\lambda(0)
	\end{split}
	\]
	because by assumption $(\nabla h)(0)=0$\,.
	In turn, (i) now reads $h\psi=F_\lambda+\frac{F_\lambda(0)}{\alpha+\frac{\sqrt{\lambda}}{4\pi}}G_\lambda$, which means, in view of the domain decomposition \eqref{eq:Ds_s3/2-2}, that $h\psi$ belongs to $\widetilde{H}^s_\alpha(\mathbb{R}^3)$. Owing to (iv) above and to the norm equivalence \eqref{eq:equiv_of_norms_s322}, we conclude
	\[
	\|h\psi\|_{\widetilde{H}^s_\alpha}\;\approx\;\|F_\lambda\|_{H^s}\,\lesssim\;\|h\|_{W^{s,q}}\|\phi_\lambda\|_{H^s}\,,
	\]
	which completes the proof.
\end{proof}

Combining Lemmas \ref{convu_bou} and \ref{produ_bou} one therefore has the trilinear estimate
\begin{equation}\label{eq:LaLh}
\|(w*(u_1u_2))u_3\|_{\widetilde{H}_{\alpha,\mathrm{rad}}^s(\mathbb{R}^3) }\;\lesssim\;\|w\|_{W^{s,p}(\R^3)}\prod_{j=1}^3\|u_j\|_{\widetilde{H}_{\alpha,\mathrm{rad}}^s(\mathbb{R}^3) }.
\end{equation}

Let us now prove Theorem \ref{highh}.

\begin{proof}[Proof of Theorem \ref{highh}]~
	
	From the expression \eqref{eq:Smap} for the solution map $\Phi(u)$ and from the bound \eqref{eq:LaLh} one finds
	\[\tag{i}
	\begin{split}
	\|\Phi(u)\|_{L^\infty\widetilde{H}_{\alpha,\mathrm{rad}}^s }\;&\leqslant\;\|f\|_{\widetilde{H}_{\alpha,\mathrm{rad}}^s }+T\,\|(w*|u|^2)u\|_{L^\infty\widetilde{H}_{\alpha,\mathrm{rad}}^s } \\
	&\leqslant\;\|f\|_{\widetilde{H}_{\alpha,\mathrm{rad}}^s }+C_1\,T\,\|w\|_{W^{s,p}}\,\|u\|^3_{L^{\infty}\widetilde{H}^s_{\alpha,\mathrm{rad}}}
	\end{split}
	\]
	for some constant $C_1>0$.
	
	Moreover,
	\begin{equation*}\tag{ii}
	\begin{split}
	\|\Phi(u&)-\Phi(v)\|_{L^\infty L^2 }\;\leqslant\;T\,\|(w*|u|^2)u-(w*|v|^2)v\|_{L^\infty L^2 } \\
	&\lesssim\;T\,\big( \|(w*|u|^2)(u-v)\|_{L^\infty L^2 } + \big\|\big(w*(|u|^2-|v|^2\big) v\big\|_{L^\infty L^2 }\big) \,.
	\end{split} 
	\end{equation*}
	For the first summand in the r.h.s.~above the bound \eqref{eq:convu_bou} and H\"older's inequality yield
	\begin{equation*}\tag{iii}
	\begin{split}
	\|(w*|u|^2)(u-v)\|_{L^\infty L^2 }\;&\leqslant\; \|w*|u|^2\|_{L^{\infty}L^{\infty}}\,\|u-v\|_{L^{\infty}L^2}\\
	&\lesssim\; \|w\|_{W^{s,p}}\,\|u\|^2_{L^{\infty}\widetilde{H}_{\alpha}^s}\,\|u-v\|_{L^{\infty}L^2}\,.
	\end{split}
	\end{equation*}
	For the second summand, let us observe preliminarily that the embedding
	\begin{equation*}\tag{iv}
	\widetilde{H}^s_\alpha(\mathbb{R}^3)\;\hookrightarrow\; L^{3,\infty}(\R^3)
	\end{equation*}
	valid for $s\in(\frac{1}{2},\frac{3}{2})$ and established in the proof of Theorem \ref{thm:high_reg} holds true even more when $s\in(\frac{3}{2},2]$.
	Then (iv) above, Sobolev's embedding $W^{s,p}(\R^3)\hookrightarrow L^3(\R^3)$, and an application of Holder's and Young's inequality in Lorentz spaces, yield
	\begin{equation*}\tag{v}
	\begin{split}
	\|(w*(|u|^2\!-\!|v|^2))&v\|_{L^\infty L^2 }\;\leqslant\; \|w*(|u|^2\!-\!|v|^2)\|_{L^{\infty}L^{6,2}}\,\|v\|_{L^{\infty}L^{3,\infty}}\\
	&\lesssim\; \|w\|_{L^3}\,\|u+v\|_{L^{\infty}L^{3,\infty}}\,\|u-v\|_{L^{\infty}L^2}\,\|v\|_{L^{\infty}L^{3,\infty}}\\
	&\lesssim\; \|w\|_{W^{s,p}}\,\|u+v\|_{L^{\infty}\widetilde{H}_{\alpha}^s}\,\|u-v\|_{L^{\infty}L^2}\,\|v\|_{L^{\infty}\widetilde{H}_{\alpha}^s}\,.
	\end{split}
	\end{equation*}
	Combining (ii), (iii), and (v) we get
	\[\tag{vi}
	\begin{split}
	d(\Phi(u),\Phi(v))\;&\leqslant\;C_2\,T\,\|w\|_{W^{s,p}}\big(\|u\|^2_{L^\infty \widetilde{H}_\alpha^s }+\|v\|^2_{L^\infty \widetilde{H}_\alpha^s} \big)\,d(u,v)
	\end{split}
	\]
	for some constant $C_2>0$.

	Thus, choosing $T$ and $M$ such that
	\[
	M\;=\;2\,\|f\|_{\widetilde{H}_{\alpha}^s}\,,\qquad T\;=\;{\textstyle\frac{1}{4}}\,\big(\max\{C_1,C_2\}\,M^2 \|w\|_{W^{s,p}}\big)^{-1}\,,
	\]
	estimate (i) reads $ \|\Phi(u)\|_{L^\infty \widetilde{H}_{\alpha,\mathrm{rad}}^s}\leqslant M$ and shows that $\Phi$ maps the space $\mathcal{X}^{(0)}_{T,M}$ into itself, whereas estimate (vi) reads $d(\Phi(u),\Phi(v))\leqslant\frac{1}{2}d(u,v)$ and shows that $\Phi$ is a contraction on $\mathcal{X}^{(0)}_{T,M}$. By Banach's fixed point theorem, there exists a unique fixed point $u\in \mathcal{X}^{(0)}_{T,M}$ of $\Phi$ and hence a unique solution $u\in \mathcal{X}^{(0)}_{T,M}$ to \eqref{eq:integral_formula_Duhamel}, which is therefore also continuous in time.

	Furthermore, by a standard continuation argument we can extend such a solution over a maximal interval for which the blow-up alternative holds  true. Also the continuous dependence on the initial data is a direct consequence of the fixed point argument. 
\end{proof}

A straightforward, yet crucial for us, consequence of Theorem \ref{highh} when $s=2$ concerns the differential meaning of the local strong solution determined so far.

\begin{corollary}[Integral and differential formulation]\label{diffsensehigh}
	Let $\alpha\geqslant 0$ and $w\in W^{2,p}(\mathbb{R}^3)$, $p\in(2,+\infty)$, a spherically symmetric potential. Assume moreover $f\in\widetilde{H}_{\alpha,\mathrm{rad}}^2(\R^3)$. Let $u\in\mathcal{C}([-T,T],\widetilde{H}^{2}_{\alpha,\mathrm{rad}}(\mathbb{R}^3))$ the unique local to the Cauchy problem \eqref{eq:Cauchy_problem_sing_Hartree} in the interval $[-T,T]$, for some $T>0$, i.e. $u$ satisfies the Duhamel formula \eqref{eq:integral_formula_Duhamel}.
	Then $u(0,\cdot)=f$ and $u$ satisfies the equation $\ii\partial_t u=-\Delta_\alpha u+ (w*|u|^2)u$ as an identity in between $L^2(\R^3)$-functions.
\end{corollary}

\begin{proof}
	The bound \eqref{eq:LaLh} shows that the non-linearity defines a map $u\mapsto (w*|u|^2)u$ that is continuous from $\widetilde{H}^{2}_{\alpha,\mathrm{rad}}(\mathbb{R}^3)$ into itself, and hence in particular it is continuous from $\widetilde{H}^{2}_{\alpha,\mathrm{rad}}(\mathbb{R}^3)$ to $L^2(\mathbb{R}^3)$. Then the thesis follows by standard fact on the theory of linear semi-groups (see \cite[Section 1.6]{cazenave}).
\end{proof}

\section{Global solutions in the mass and in the energy space}\label{sec:global}

In order to study the global solution theory of the Cauchy problem \eqref{eq:Cauchy_problem_sing_Hartree} when $s=0$ (the mass space $L^2(\mathbb{R}^3)$) and $s=1$ (the energy space $\widetilde{H}^1_\alpha(\mathbb{R}^3)$), we introduce the following two quantities, that are formally conserved in time along the solutions.

\begin{definition}~
	\begin{itemize}
		\item[(i)] Let $u\in L^2(\R^3)$. We define the \emph{mass} of $u$ as 
		$$\mathcal{M}(u)\;:=\;\|u\|_{L^2}^2\,.$$
		\item[(ii)] Let $\lambda>0$ and let $u=\phi_\lambda+\kappa_\lambda\,G_{\lambda}\in \widetilde{H}_{\alpha}^1(\R^3)$, according to \eqref{eq:Ds_s1/2-3/2}. We define the \emph{energy} of $u$ as
		\begin{equation*}
		\begin{aligned}
		\mathcal{E}(u)\;:=&\;\;{\textstyle\frac12}(-\Delta_{\alpha})[u]+{\textstyle\frac14}\int_{\R^3}\!(w*|u|^2)|u|^2\,\ud x\\
		=&\;\;{\textstyle\frac12}\big(\lambda\|\phi_\lambda\|_{L^2}^2+\|\nabla\phi_\lambda\|_{L^2}^2+\big(\alpha+\textstyle\frac{\sqrt{\lambda}}{4\pi}\big)|\kappa_\lambda|^2-\lambda\|u\|_{L^2}^2\big)\\
		&\qquad+{\textstyle\frac14}\!\int_{\R^3}(w*|u|^2)|u|^2\,\ud x\,.
		\end{aligned}
		\end{equation*}
	\end{itemize}
\end{definition}

\begin{remark}
	For given $u$, the value of $(-\Delta_\alpha)[u]$ (the quadratic form of $-\Delta_\alpha$) is independent of $\lambda$, and so too is the energy $\mathcal{E}(u)$.
\end{remark}

We shall establish suitable conservation laws in order to prolong the local solution globally in time.
The mass is conserved in  $L^2(\mathbb{R}^3)$ in the following sense.

\begin{proposition}[Mass conservation in $L^2(\mathbb{R}^3)$]\label{mass_cons}~
	
	\noindent Let $\alpha\geqslant 0$, and let $w$ belong either to the class $L^\infty(\mathbb{R}^3)\cap W^{1,3}(\mathbb{R}^3)$ or to the class $w\in L^{\frac{3}{\gamma},\infty}(\R^3)$, for $\gamma\in(0,\frac32)$. For a given $f\in L^2(\mathbb{R}^3)$, let $u$ be the unique local solution in $\mathcal{C}((-T_*,T^*),L^2(\mathbb{R}^3))$ to the Cauchy problem \eqref{eq:integral_formula_Duhamel} in the maximal interval $(-T_*,T^*)$, as given by Theorem \ref{thm:high_reg}.
	Then $\mathcal{M}(u(t))$ is constant for $t\in(-T_*,T^*)$.
\end{proposition}

\begin{proof}
	Let us discuss first the case $w\in L^\infty(\mathbb{R}^3)\cap W^{1,3}(\mathbb{R}^3)$. Consider preliminarily an initial data $f\in \widetilde{H}_{\alpha}^1(\R^3)$. Owing to Corollary \ref{soddire}, for each $t\in (-T_*,T^*)$ $u$ satisfies $\ii\partial_t u=-\Delta_\alpha u+ (w*|u|^2)u$ as an identity between $\widetilde{H}^{-1}_\alpha$-functions, whence
	\begin{equation*}
	\big\langle \,\ii\partial_t u+\Delta_\alpha u-(w*|u|^2)u,u\big\rangle_{\widetilde{H}^{-1}_\alpha,\widetilde{H}^{1}_\alpha}\;=\;0\,.
	\end{equation*}
	The imaginary part of the above identity gives
	$${\textstyle\frac{\ud}{\ud t}}\,\|u(t)\|_{L^2}^2\;=\;0\,,$$
	which implies that $\mathcal{M}(u(t))$ is constant on $(-T_*,T^*)$.
	For arbitrary $f\in L^2(\mathbb{R}^3)$ we use a density argument. Let $(f_n)_n$ be a sequence in $\widetilde{H}_{\alpha}^1(\R^3)$ such that  $f_n\xrightarrow[]{n\to\infty} f$ in  $L^2(\mathbb{R}^3)$, and denote by $u_n$ the solution to the Cauchy problem \eqref{eq:Cauchy_problem_sing_Hartree} with initial datum $f_n$. Because of the continuous dependence on the initial data, we have that $u_n\to u$ in $\mathcal{C}(I,L^2(\R^3))$, for every closed interval $I\subset(-T_*,T^*)$. Since $\mathcal{M}(u_n(t))=\mathcal{M}(u_n(0))=\mathcal{M}(f_n)$ for every $n$, we deduce that $\mathcal{M}(u(t))=\mathcal{M}(f)$ for $t\in I$. Owing to the continuity of the map $t\mapsto\mathcal{M}(u(t))$, we conclude that $\mathcal{M}(u(t))=\mathcal{M}(f)$ for $t\in(-T_*,T^*)$.
	
	Let us discuss now the case $w\in L^{\frac{3}{\gamma},\infty}(\R^3)$, $\gamma\in(0,\frac32)$. Consider preliminarily an initial data $f\in\widetilde{H}_{\alpha}^1(\R^3)$ and a Schwartz potential $w$. Owing to Corollary \ref{soddire}, for each $t\in (-T_*,T^*)$ $u$ satisfies $\ii\partial_t u=-\Delta_\alpha u+ (w*|u|^2)u$ as an identity between $\widetilde{H}^{-1}_\alpha$-functions, and reasoning as above we deduce that $\mathcal{M}(u(t))$ is constant on $(-T_*,T^*)$. For arbitrary $f\in L^2(\mathbb{R}^3)$ and $w\in L^{\frac{3}{\gamma},\infty}(\R^3)$, $\gamma\in(0,\frac32)$, we use a density argument. Let $(f_n)_n$ be a sequence in $\widetilde{H}_{\alpha}^1(\R^3)$ such that  $f_n\xrightarrow[]{n\to\infty} f$ in  $L^2(\mathbb{R}^3)$, $(w_n)_n$ be a sequence of Schwartz potentials such that  $w_n\xrightarrow[]{n\to\infty} w$ in  $L^{\frac{3}{\gamma},\infty}(\R^3)$, and denote by $u_n$ the $L^2$ strong solution to the Cauchy problem \eqref{eq:Cauchy_problem_sing_Hartree} with initial datum $f_n$ and potential $w_n$. The stability result given by Proposition \ref{stabint_mass} guarantees that $u_n\xrightarrow[]{n\to +\infty} u$ in $\mathcal{C}([-T,T],L^2(\R^3))$ for some $T>0$, whence $\mathcal{M}(u_n(t))\xrightarrow[]{n\to +\infty}\mathcal{M}(u(t))$ for $t\in[-T,T]$. Using the mass conservation for $u_n$  we deduce that $\mathcal{M}(u(t))=\mathcal{M}(f)$ for $t\in[-T,T]$.
	Repeating the above argument with $f$ replaced by $u(t_0)$ for some $t_0\in(-T_*,T^*)$ yields the property that $t\mapsto\mathcal{M}(u(t))$ is constant in a suitable interval around $t_0$ and hence, by the arbitrariness of $t_0$, it is locally constant on the whole $(-T_*,T^*)$. But $(-T_*,T^*)\ni t\mapsto\mathcal{M}(u(t))$ is also continuous, whence the conclusion.
\end{proof}

We therefore conclude the following.

\begin{proof}[Proof of Theorem \ref{thm:GWP_Mass}]
	An immediate consequence of the conservation of the mass, i.e., conservation of the $L^2$-norm, and of the blow up alternative in $L^2$.
\end{proof}

Let us move now to the conservation of mass and energy in the energy space. We observe the following.

\begin{lemma}\label{preli}
	Let $\alpha\geqslant 0$ and let $w\in W^{1,p}(\mathbb{R}^3)$ for some $p>2$. If $v_n\xrightarrow[]{n\to +\infty} v$ in $\widetilde{H}_{\alpha}^1(\R^3)$, then $\mathcal{E}(v_n)\xrightarrow[]{n\to +\infty}\mathcal{E}(v)$. As a consequence, if $u\in\mathcal{C}([-T,T],\widetilde{H}_{\alpha}^1(\R^3))$ for some $T>0$, then $t\mapsto\mathcal{E}(u(t))$ is continuous on $[-T,T]$.
\end{lemma}

\begin{proof}
	The limit $\mathcal{E}(v_n)\xrightarrow[]{n\to +\infty}\mathcal{E}(v)$ follows from the inequality
	\[
	\begin{split}
	|\mathcal{E}(v)-\mathcal{E}(v_n)|\;&\lesssim\;|\,(-\Delta_{\alpha})[v]-(-\Delta_{\alpha})[v_n]\,| \\
	&\qquad +\|(w*|v|^2)|v|^2-(w*|v_n|^2)|v_n|^2\|_{L^1}
	\end{split}
	\]
	combined with the estimates
	\[
	|\,(-\Delta_{\alpha})[v]-(-\Delta_{\alpha})[v_n]\,|\;\lesssim\;\|v-v_n\|_{\widetilde{H}_{\alpha}^1}\big(\|v\|_{\widetilde{H}_{\alpha}^1}+\|v_n\|_{\widetilde{H}_{\alpha}^1}\big)
	\]
	and
	\[
	\begin{split}
	\|(w*|v|^2)&|v|^2-(w*|v_n|^2)|v_n|^2\|_{L^1} \\
	&\lesssim\;\|(w*|v|^2)(|v|^2-|v_n|^2)\|_{L^1}+\|(w*(|v|^2-|v_n|^2)|v_n|^2\|_{L^1} \\
	&\lesssim\;\|w*|v|^2\|_{L^\infty}\|v-v_n\|_2\big(\|v\|_2+\|v_n\|_2\big) \\
	&\qquad\qquad +\big\||w|*\big(|v-v_n|(|v|+|v_n|)\big)\big\|_{L^\infty}\|v_n\|_{L^2}^2 \\
	&\lesssim\;\|w\|_{W^{1,p}}\|v-v_n\|_{\widetilde{H}_{\alpha}^1}\big(\|v\|^2_{\widetilde{H}_{\alpha}^1}+\|v_n\|^2_{\widetilde{H}_{\alpha}^1}\big)\,,
	\end{split}
	\]
	the last two steps above following from H\"{o}lder's and Young's inequality, and from the inequality \eqref{eq:convu}.
\end{proof}

We then see that mass and energy are conserved in the spherically symmetric component of the energy space.

\begin{proposition}[Mass and energy conservation in $\widetilde{H}^{1}_{\alpha,\mathrm{rad}}(\mathbb{R}^3)$]\label{enel_cons}~
	
	\noindent Let $\alpha\geqslant 0$. For a given $w\in W^{1,p}_{\mathrm{rad}}(\mathbb{R}^3)$, $p\in(2,+\infty)$, and a given $f\in\widetilde{H}_{\alpha,\mathrm{rad}}^1(\R^3)$, let $u$ be the unique local solution in $\mathcal{C}((-T_*,T^*),\widetilde{H}^{1}_{\alpha,\mathrm{rad}}(\mathbb{R}^3))$ to the Cauchy problem \eqref{eq:integral_formula_Duhamel} in the maximal interval $(-T_*,T^*)$, as given by Theorem \ref{thm:high_reg}.
	Then $\mathcal{M}(u(t))$ and $\mathcal{E}(u(t))$ are constant for $t\in(-T_*,T^*)$.
\end{proposition}

\begin{proof}
	We start proving the statement for the mass. 
	Owing to Corollary \ref{soddire}, for each $t\in (-T_*,T^*)$ $u$ satisfies $\ii\partial_t u=-\Delta_\alpha u+ (w*|u|^2)u$ as an identity in $\widetilde{H}^{-1}_\alpha(\mathbb{R}^3)$, whence
	\begin{equation*}
	\big\langle \,\ii\partial_t u+\Delta_\alpha u-(w*|u|^2)u,u\big\rangle_{\widetilde{H}^{-1}_\alpha,\widetilde{H}^{1}_\alpha}\;=\;0\,.
	\end{equation*}
	The imaginary part of the above identity gives
	$${\textstyle\frac{\ud}{\ud t}}\,\|u(t)\|_{L^2}^2\;=\;0\,,$$
	which implies that $\mathcal{M}(u(t)$ is constant on $(-T_*,T^*)$.

	Let us prove now that the energy is conserved, first in the special case  $f\in \widetilde{H}_{\alpha,\mathrm{rad}}^2(\R^3)$ and $w\in W_{\mathrm{rad}}^{2,p}(\mathbb{R}^3)$, for $p\in(2,+\infty)$. Owing to Corollary \ref{diffsensehigh}, $u$ satisfies $\ii\partial_t u=-\Delta_\alpha u+ (w*|u|^2)u$ as an identity in $L^2(\mathbb{R}^3)$, whence
	\begin{equation*}
	\big\langle \,\ii\partial_t u+\Delta_\alpha u-(w*|u|^2)u,\partial_t u\,\big\rangle_{L^2}\;=\;0\,.
	\end{equation*}
	The real part in the above identity gives
	$${\textstyle\frac{\ud}{\ud t}}\big({\textstyle\frac12}\langle-\Delta_{\alpha}u,u\rangle_{L^2}-{\textstyle\frac14}\!\textstyle\int_{\R^3}(w*|u|^2)|u|^2\ud x\big)\;=\;0\,,$$
	which implies that $\mathcal{E}(u(t))$ is constant on $(-T_*,T^*)$.

	For arbitrary $f\in\widetilde{H}_{\alpha,\mathrm{rad}}^1(\R^3)$ and $w\in W_{\mathrm{rad}}^{1,p}(\mathbb{R}^3)$ we use the stability result of Proposition \ref{stabint}. Let $(f_n)_n$ be a sequence in $\widetilde{H}_{\alpha,\mathrm{rad}}^2(\R^3)$ and $(w_n)_n$ be a sequence in $W_{\mathrm{rad}}^{2,p}(\R^3)$ such that $f_n\xrightarrow[]{n\to +\infty}f$ in $\widetilde{H}_{\alpha}^1(\R^3)$ and $w_n\xrightarrow[]{n\to +\infty} w$ in $W^{1,p}(\mathbb{R}^3)$, and denote by $u_n$ the solution to the Cauchy problem \eqref{eq:Cauchy_problem_sing_Hartree} with initial datum $f_n$ and potential $w_n$. Then Proposition \ref{stabint} guarantees that $u_n\xrightarrow[]{n\to +\infty} u$ in $\mathcal{C}([-T,T],\widetilde{H}_{\alpha}^1(\R^3))$ for some $T>0$, and Lemma \ref{preli} implies that $\mathcal{E}(u_n(t))\xrightarrow[]{n\to +\infty}\mathcal{E}(u(t))$ for $t\in[-T,T]$. Using the energy conservation for $u_n$  we deduce that $\mathcal{E}(u(t))=\mathcal{E}(f)$ for $t\in[-T,T]$.
	Repeating the above argument with $f$ replaced by $u(t_0)$ for some $t_0\in(-T_*,T^*)$ yields the property that $t\mapsto\mathcal{E}(u(t))$ is constant in a suitable interval around $t_0$ and hence, by the arbitrariness of $t_0$, it is locally constant on the whole $(-T_*,T^*)$. But $(-T_*,T^*)\ni t\mapsto\mathcal{E}(u(t))$ is also continuous, whence the conclusion.
\end{proof}

We are now ready to prove our result on the solution theory for the Cauchy problem \eqref{eq:Cauchy_problem_sing_Hartree}. 

\begin{proof}[Proof of Theorem \ref{thm:GWP}]~
	
	Let $u\in\mathcal{C}((-T_*,T^*),\widetilde{H}^{1}_{\alpha,\mathrm{rad}}(\mathbb{R}^3))$ be the unique local strong solution to \eqref{eq:Cauchy_problem_sing_Hartree}, on the maximal time interval $(-T_*,T^*)$, with given initial datum $f=\phi_\lambda+c\, G_\lambda\in \widetilde{H}^{1}_{\alpha,\mathrm{rad}}(\mathbb{R}^3) $, for some $\lambda>0$, and given potential $w\in W^{1,p}_{\mathrm{rad}}(\mathbb{R}^3)$, for some $p\in(2,+\infty)$, as provided by Theorem \ref{thm:high_reg}. Then $(-T_*,T^*)\ni t\mapsto \mathcal{M}(u(t))+\mathcal{E}(u(t))$ is the constant map, as follows from Propositions \ref{enel_cons}. Decomposing $u(t)=\phi_\lambda(t)+\kappa_\lambda(t)G_\lambda$ for each $t\in(-T_*,T^*)$ and using \eqref{eq:equiv_of_norms_s1232} we find 
	\begin{align*}
	\|u(t)\|^2_{\widetilde{H}_{\alpha}^1}\;&\approx\;\|\phi_\lambda(t)\|_{H^1}^2+|\kappa_\lambda(t)|^2 \\
	&\lesssim\; (\lambda+1)\,\|u(t)\|_{L^2} \nonumber \\
	&\qquad+\big(\lambda\|\phi_\lambda(t)\|_{L^2}^2-\lambda\|u(t)\|_{L^2}^2+\|\nabla\phi_\lambda(t)\|_{L^2}^2+\big(\alpha+\textstyle\frac{\sqrt{\lambda}}{4\pi}\big)|\kappa_\lambda(t)|^2\big) \nonumber \\
	&\lesssim\;\mathcal{M}(u(t))+{\textstyle\frac12}(-\Delta_{\alpha})[u(t)]\,. \tag{*}
	\end{align*}

	For part (i) of the statement, we observe that
	\begin{align*}
	\sup_{t\in(-T_*,T^*)}\|u(t)\|^2_{\widetilde{H}_{\alpha}^1}\;&\lesssim\; \sup_{t\in(-T_*,T^*)}\big(\mathcal{M}(u(t))+{\textstyle\frac12}(-\Delta_{\alpha})[u(t)]\big)\\
	&\lesssim\;\sup_{t\in(-T_*,T^*)} \big(\mathcal{M}(u(t))+\mathcal{E}(u(t))+\|(w*|u(t)|^2)|u(t)|^2\|_{L_x^1}\big)\\
	&\lesssim\; 1+\sup_{t\in(-T_*,T^*)}\|w*|u|^2\|_{L_x^{\infty}}\,\|u\|^2_{L_x^2}\\
	&\lesssim\; 1+\sup_{t\in(-T_*,T^*)}\|w\|_{W^{s,p}}\|u\|^2_{\widetilde{H}_{\alpha}^1}\,\|f\|^2_{L^2}\,,
	\end{align*}
	having used (*), the estimate \eqref{eq:convu}, and the mass and energy conservation. Therefore, if $\|f\|_{L^2}$ is sufficiently small (depending only on $\|w\|_{W^{s,p}}$), then
	$$\sup_{t\in(-T_*,T^*)}\|u(t)\|^2_{\widetilde{H}_{\alpha}^1}\;\lesssim\; 1\,,$$
	and we conclude that solution is global, owing to the blow up alternative.

	For part (ii) of the statement, the additional assumption $w\geqslant 0$ implies
	$${\textstyle\frac12}(-\Delta_{\alpha})[u(t)]\;\leqslant \;{\textstyle\frac12}(-\Delta_{\alpha})[u(t)]+{\textstyle\frac14}\int_{\R^3}(w*|u(t)|^2)|u(t)|^2\ud x\;=\;\mathcal{E}(u(t))\,,$$
	which, combined with (*) and the mass and energy conservation yields
	\begin{equation*}
	\sup_{t\in(-T_*,T^*)}\|u(t)\|^2_{\widetilde{H}_{\alpha}^1}\;\lesssim\; \sup_{t\in(-T_*,T^*)}\big(\mathcal{M}(u(t))+\mathcal{E}(u(t))\big)\;\lesssim\; 1\,.
	\end{equation*}
	Therefore, the solution is global, by the blow up alternative. Since this is true for every initial datum $f\in \widetilde{H}^{1}_{\alpha,\mathrm{rad}}(\mathbb{R}^3)$, we deduce global well-posedness for \eqref{eq:Cauchy_problem_sing_Hartree}.
\end{proof}

\section{Comments on the spherically symmetric solution theory}\label{sec:spheric}

As initially mentioned in the Introduction and then shown in the preceding discussion, part of the solution theory was established for spherically symmetric potentials and solutions (Theorems \ref{highh} and \ref{thm:GWP}) and in this Section we collect our remarks on the emergence of such a feature.

This is indeed a natural phenomenon both for the local high regularity theory and for the global theory in the energy space, as we are now going to explain. Of course, the spherically symmetric solution theory is the most relevant in the study of the singular Hartree equation, since the linear part, namely the operator $-\Delta_\alpha$, differs from the ordinary $-\Delta$ precisely in the $L^2$-sector of rotationally symmetric functions.

For the local theory, one ineludible ingredient of the fixed point argument is the treatment of the non-linear part of the solution map \eqref{eq:Smap} with a $\widetilde{H}^s_\alpha$-estimate that we close by means of the trilinear estimate \eqref{eq:LaL}/\eqref{eq:LaLh}.

This estimate is designed for functions of the form $hu$, where $h=w*|u|^2$, and it is crucially sensitive to the specific structure of the space $\widetilde{H}^s_\alpha(\mathbb{R}^3)$ for $s>\frac{1}{2}$ (Theorem \ref{thm:domain_s}(ii)-(iii)). In particular, in order to recognise that the regular part $hu$ is indeed a $H^s$-function, one must show that $(h-h(0))G_\lambda\in H^s(\mathbb{R}^3)$. Technically this is dealt with by means of the fractional Leibniz rule, suitably generalised so as to avoid the direct $L^p$-estimate of $s$ derivatives of each factor $h-h(0)$ and $G_\lambda$; already heuristically it is clear that this only works with a sufficient vanishing rate of $h-h(0)$ as $|x|\to 0$ in order to compensate the local singularity of $G_\lambda$.

For intermediate regularity (Proposition \ref{tst} and Lemmas \ref{convu}-\ref{produ}) the vanishing rate $h(x)-h(0)\sim |x|^\theta$ that can be deduced from the embedding $w*|u^2|\in \mathcal{C}^{0,\theta}(\mathbb{R}^3)$ is enough to close the argument and no spherical symmetry is required. For high regularity (Proposition \ref{tstdue} - Lemma \ref{le:fra_decay}, and Lemmas \ref{convu_bou}-\ref{produ_bou}) the embedding $w*|u^2|\in \mathcal{C}^{1,\theta}(\mathbb{R}^3)$ would only guarantee an insufficient vanishing rate $h(x)-h(0)\sim |x|$; since one needs $h(x)-h(0)\sim |x|^{1+\theta}$, this requires the additional condition $\nabla h(0)=0$. For the latter condition to hold for $h=w*|u|^2$, as shown in the proof of Lemma \ref{convu_bou}(ii), the spherical symmetry of both $w$ and $u$ appears as the most natural and explicitly treatable assumption.

In fact, the condition $\nabla h(0)=0$ is even more crucial and apparently unavoidable in one further point of the argument $hu\in\widetilde{H}^s_\alpha(\mathbb{R}^3)$, because unlike the intermediate regularity case, where it suffices to prove that the regular component of $hu$ is a $H^s$-function, in the high regularity case one must also prove that such regular component satisfies the correct \emph{boundary condition} in connection with the singular component. As shown in the proof of Lemma \ref{produ_bou}, the correct boundary condition is \emph{equivalent} to $|x|^{-1}(h(x)-h(0))\to 0$ as $|x|\to 0$, for which $\nabla h(0)=0$ is again necessary.

Concerning the global theory in the energy space, the emergence of a solution theory for spherically symmetric functions is due to one further mechanism. As usual, globalisation is based upon the mass and energy conservation. In the theory of semi-linear Schr\"{o}dinger equations it is typical that the conservation laws are deduced from a suitably regularised problem (see, e.g., the proof of \cite[Theorem 3.3.5]{cazenave}). In the present context (Proposition \ref{enel_cons}) we follow this scheme showing first the conservation laws at the level of $\widetilde{H}^2_\alpha$-regularity, and then controlling the stability of a density argument which is set for $\widetilde{H}^1_\alpha$-regularity. Clearly the first step appeals to the local $\widetilde{H}^2_\alpha$-theory, which is derived only for the spherically symmetric case, thus the stability argument can only work in the spherically symmetric sector of the energy space.

\section*{Acknowledgements}

We are deeply indebted to G.~Dell'Antonio for many enlightening discussions on the subject and to V.~Georgiev for his precious advices on his work \cite{Fujiwara-Georgiev-Ozawa-2016} and on Theorem \ref{KatoPonce_Vladimir}.



\def\cprime{$'$}

\end{document}